%% file: cobosa.tex
\documentclass[11pt]{amsart}
\usepackage{amsmath}
\usepackage{amssymb,amscd}
\usepackage{textcomp}
\usepackage{graphicx}
\usepackage{color}
\usepackage[colorlinks=false, urlcolor=blue,bookmarks=true,bookmarksopen=true, citecolor=blue,hypertex]{hyperref}

\numberwithin{equation}{section}

\newtheorem{defn}{Definition}[section]
\newtheorem{theorem}{Theorem}[section]

\newtheorem{example}[theorem]{Example}

\newtheorem{lemma}[theorem]{Lemma}
\newtheorem{prop}[theorem]{Proposition}
\newtheorem{remark}[theorem]{Remark}
\newtheorem{qus}[theorem]{Question}

\theoremstyle{definition}
\newtheorem{conn}{Connected sum}

\def \begineq{\begin{equation}}
\def \endeq{\end{equation}}

\def \bb{\mathbb}

\def \CC{{\bb{C}}}
\def \CP{{\bb{CP}}}

\def \RR{{\bb{R}}}
\def \TT{{\bb{T}}}

\def \ZZ{{\bb{Z}}}

\def \({\left(}
\def \){\right)}
\def \<{\langle}
\def \>{\rangle}
\def \bar{\overline}

\begin{document}

\title{$\TT^2$-cobordism of Quasitoric $4$-Manifolds}

\author[S. Sarkar]{Soumen Sarkar}

\address{Theoretical Statistics and Mathematics Unit, Indian
Statistical Institute, 203 B. T. Road, Kolkata 700108, India}

\email{soumensarkar20@gmail.com}



\abstract We show the $\TT^2$-cobordism group of the category of $4$-dimensional quasitoric manifolds
is generated by the $\TT^2$-cobordism classes of $\CP^2$. We construct nice oriented $\TT^2$ manifolds
with boundary where the boundary is the Hirzebruch surfaces. The main tool is the theory of quasitoric
manifolds. 
\endabstract

\maketitle

\section{Introduction}\label{intro}
Cobordism was explicitly introduced by Lev Pontryagin in geometric work on manifolds. In the early 1950's
Ren\'{e} Thom \cite{Tho} showed that cobordism groups could be computed by results of homotopy theory
using the Thom complex construction.
The non-oriented, oriented and complex cobordism cobordism rings are completely determined. Since the
Thom transversality theorem does not hold in equivariant category, the results (like non-equivariant case)
can not be reduced to homotopy theory. The equivariant cobordism has many development, but the equivariant
cobordism ring is not determined for any group. We consider the following category: the objects are all
quasitoric manifolds and morphisms are torus equivariant maps between quasitoric manifolds. Here by torus
we mean compact torus $\TT^n := U(1)^n = (\ZZ^n \otimes \RR)/ \ZZ^n$ of dimension $n$.
We compute the $\TT^2$-cobordism group of $4$-dimensional manifolds in this category.
We show that the $\TT^2$-cobordism group of the category of $4$-dimensional quasitoric manifolds
is generated by the $\TT^2$-cobordism classes of $\CP^2$. The main tool is the theory of quasitoric manifolds.

Quasitoric manifolds and small covers were introduced by Davis and Januskiewicz in \cite{DJ}.
A manifold with quasitoric (small cover) boundary is a manifold with boundary where the boundary
is a disjoint union of some quasitoric manifolds (respectively small covers).

Following \cite{OR} we discuss the definition of quasitoric manifolds and the classification of $4$-dimensional
quasitoric manifolds in Section \ref{defegm}. This classification is needed to prove the Lemma \ref{coborlm}.
In Section \ref{poly} we introduce $edge$-$simple ~ polytopes$ and study their properties. 
We give the brief definition of some manifolds with quasitoric and small cover boundary in a constructive way in Section \ref{def}.
There is a natural torus action on these manifolds with quasitoric boundary having a simple convex polytope as the orbit space.
The fixed point set of the torus action on the manifold with quasitoric boundary
corresponds to the disjoint union of closed intervals of positive length.
Interestingly, we show that such a manifold with quasitoric boundary could be viewed as the quotient space of a quasitoric manifold
corresponding to a certain circle action on it. This is done in the subsection \ref{quot}.

In Section \ref{ori} we show these manifolds with quasitoric boundary are orientable and compute
their Euler characteristic. In Section \ref{cobort2} we show that the $\TT^2$-cobordism group of $4$-dimensional quasitoric manifolds
is generated by the $\TT^2$-cobordism classes of the complex projective space $\CP^2$, see Lemma \ref{coborlm}.
We construct nice oriented $\TT^2$ manifolds with boundary where the boundary is the Hirzebruch surfaces. In particular,
$\TT^2$-cobordism class of a Hirzebruch surface is trivial, see Lemma \ref{hizb}. In Theorem \ref{coborthm}
we compute a set of generators of the $\TT^2$-cobordism group of $4$-dimensional quasitoric manifolds.

\section{Quasitoric manifolds}\label{defegm}
An $n$-dimensional simple polytope in $\RR^n$ is a convex polytope where exactly $n$ bounding hyperplanes meet at 
each vertex. The codimension one faces of a convex polytope are called $facets$.
Let $\mathcal{F}(P)$ be the set of facets of an $n$-dimensional simple polytope $P$.
Following \cite{BP} we give definition of quasitoric manifold, characteristic function and classification.
\begin{defn}
A smooth action of $\TT^n$ on a $2n$-dimensional smooth manifold $M$ is said to be locally
standard if every point $y \in M $ has a $\TT^n$-stable open neighborhood
$U_y$ and a diffeomorphism $\psi : U_y \to V$, where $V$ is a $\TT^n$-stable
open subset of $\CC^n$, and an isomorphism $\delta_y : \TT^n \to \TT^n$ such that
\ $\psi (t\cdot x) = \delta_y (t) \cdot \psi(x)$ for all $(t,x) \in \TT^n \times U_y$.
\end{defn}

\begin{defn}\label{qtd02}
A closed smooth $2n$-dimensional $\TT^n$-manifold $M$ is called a quasitoric manifold over $P$ if
the following conditions are satisfied:
\begin{enumerate}
\item the $\TT^n$ action is locally standard,
\item there is a projection map $\mathfrak{q}: M \to P$ constant on $\TT^n$ orbits which maps every
$l$-dimensional orbit to a point in the interior of a codimension-$l$ face of $P$.
\end{enumerate}
\end{defn}

All complex projective spaces $\CP^{n}$ and their equivariant connected sums, products are quasitoric manifolds.
\begin{lemma}[\cite{DJ}, Lemma 1.4]\label{clema}
Let $\mathfrak{q}: M \to P$ be a $2n$-dimensional quasitoric manifold over $P$. There is a projection map
$f: \TT^n \times P \to M$ so that for each $q \in P$, $f$ maps $\TT^n \times q$ onto $ \mathfrak{q}^{-1}(q)$.
\end{lemma}

Define an equivalence relation $\sim_2$ on $\ZZ^n$ by $x \sim_2 y$ if and only if $y = \pm x$.
Denote the equivalence class of $x$ in the quotient space $\ZZ^n/\ZZ_2$ by $[x]$.
\begin{defn} \label{charfun}
A function $\eta : \mathcal{F}(P) \to \ZZ^n/\ZZ_2$ is called characteristic function if the submodule generated by
$\{\eta(F_{j_1}), \ldots, \eta(F_{j_l})\}$ is an $l$-dimensional direct summand of $\ZZ^n$ whenever
the intersection of the facets $F_{j_1}, \ldots, F_{j_l}$ is nonempty.

The vectors $ \eta(F_{j})$ are called characteristic vectors and the pair $(P, \eta)$ is called a characteristic pair.
\end{defn}

In \cite{DJ} the authors show that we can construct a quasitoric manifold from the pair $(P, \eta)$.
Also given quasitoric manifold we can associate a characteristic pair to it up to choice of signs of
characteristic vectors. For simplicity of notations we may write the images of
characteristic and isotropy functions by their class representative. The isotropy function is defined in Section \ref{def}.

\begin{defn}
 Two actions of $\TT^n$ on $2n$-dimensional quasitoric manifolds $M_1$ and $M_2$ are called equivalent
if there is a homeomorphism $ f : M_1 \to M_2$ such that $f(t \cdot x) = t \cdot f(x)$ for all
$(t, x ) \in \TT^n \times M_1$. 
\end{defn}

\begin{defn}\label{defdel}
Let $\delta : \TT^n \to \TT^n$ be an automorphism. Two quasitoric manifolds $M_1$ and $M_2$ over the same
polytope $P$ are called $\delta$-equivariantly homeomorphic if there is a homeomorphism $ f : M_1 \to M_2$
such that $f(t \cdot x) = \delta(t)\cdot f(x)$ for all $(t, x ) \in \TT^n \times M_1$.

When $\delta$ is the identity automorphism, $f$ is called an equivariant homeomorphism.
\end{defn}

\begin{lemma}[\cite{DJ}, Proposition 1.8]\label{clema2}
Let $\mathfrak{q}:M \to P$ be a $2n$-dimensional quasitoric manifold over $P$ and
$\eta:\mathcal{F}(P) \to \ZZ^n/\ZZ_2$ be its associated characteristic function. Let
$\mathfrak{q}_{M}: M(P, \eta) \to P$ be the quasitoric manifold constructed from the pair $(P, \eta)$.
Then the map $f: \TT^n \times P \to M$ of Lemma \ref{clema} descends to an equivariant homeomorphism
$ M(P, \eta) \to M$ covering the identity on $P$.
\end{lemma}

The automorphism $\delta$ of Definition \ref{defdel} induces an automorphism $\delta_{\ast}$
of the poset of subtori of $\TT^n$ or equivalently, an automorphism $\delta_{\ast}$
of the poset of submodules of $\ZZ^n$. This automorphism descends to a $\delta$-$translation$ of
characteristic pairs, in which the two characteristic functions differ by $\delta_{\ast}$.
Using Lemma \ref{clema} and \ref{clema2} we can prove the following Proposition.

\begin{prop}[\cite{BP}, Proposition 5.14]\label{probi}
There is a bijection between $\delta$-equivariant homeomorphism classes of quasitoric manifolds and
$\delta$-translations of characteristic pairs $(P, \eta)$.
\end{prop}
\begin{remark}\label{equiclas}
 Suppose $\delta$ is the identity automorphism of $\TT^n$. From Proposition \ref{probi} we have two quasitoric
manifolds are equivariantly homeomorphic if and only if their characteristic functions are the same.  
\end{remark}
\begin{remark}\label{orientq}
 A quasitoric manifold $M$ over $P$ is simply connected. So $M$ is orientable.
A choice of orientation on $\TT^n $ and $ P$ gives an orientation on $M$.
In this article we fix the positive orientation on $\TT^n$. The orientation
on the circle subgroup determined by the vectors $\eta(F_j)$ is the induced
orientation of $\TT^n$. So an orientation of $P$ determines an orientation
of the corresponding quasitoric manifolds. 
\end{remark}
\begin{conn}
Equivariant connected sum of oriented quasitoric manifolds is discussed explicitly in
section $6$ of \cite{BR}. We discuss the equivariant connected sum of quasitoric manifolds
briefly following \cite{DJ} and \cite{BR}.
Let $\mathfrak{q}_1: M_1 \to P_1$ and $\mathfrak{q}_2: M_2 \to P_2$ be two $2n$-dimensional
oriented quasitoric manifolds over $P_1$ and $P_2$ respectively. Let $x_1 \in M_1$ and $x_2 \in M_2$
be two fixed points. Changing the action (if necessary) of $\TT^n$ on $M_2$ by an automorphism of
$\TT^n$, we can assume that $\TT^n$ actions on a $\TT^n$ invariant neighborhood $U_{1}$ of $x_1$ and
$U_2$ of $x_2$ are equivalent. Let $B_1 \subseteq U_1$ and $B_2 \subseteq U_2$ be two invariant open
ball around $x_1$ and $x_2$ respectively. Identifying the boundary spheres of $M_1 - B_1$ and $M_2 - B_2$ via an orientation
reversing (with respect to the induced orientation) equivariant diffeomorphism we get a smooth manifold,
denoted by $M_1 \# M_2$, with a natural locally standard $\TT^n$ action. The orbit space $P_1 \# P_2$ of
this action can be described as follows. Let $\mathfrak{q}_1(x_1)=v_1$ and $\mathfrak{q}_2(x_2)=v_2$
be the corresponding vertices in $P_1$ and $P_2$ respectively. Delete a neighborhood $\bigtriangleup_{v_1}$ of $v_1$ in $P_1$
such that the closer of $\bigtriangleup_{v_1}$ in $P_1$ is diffeomorphic to the $n$-simplex. Let $P_1^{\prime}$ be
the resulting polytope. Then $P_1^{\prime}$ has a new facet $\bigtriangleup^{n-1}(v_1)$ which is an $(n-1)$-simplex.
Similarly we construct the polytope $P_2^{\prime}$ from $P_2$. Let $F^i_1, F^i_2, \ldots, F^i_n$ be the facets
meeting at $v_i$ of $P_i$. Since the actions of $\TT^n$ in a neighborhood of $x_1$ and $x_2$ are equivalent,
we may assume that the characteristic vector of $F_j^1$ and $F_j^2$ are same for $j=1, 2, \ldots, n$.
We can obtain the space $P_1 \# P_2$ by gluing the polytopes $P^{\prime}_1$ and $P^{\prime}_2$ along
$\bigtriangleup^{n-1}(v_1)$ and $\bigtriangleup^{n-1}(v_2)$ so that $F^1_j$ and $F^2_j$ make a new
facet for $j = 1, \ldots, n$. Then $M_1 \# M_2$ is an oriented quasitoric manifold over $P_1 \# P_2$.
The manifold $M_1 \# M_2$ is called the equivariant connected sum of $M_1$ and $M_2$.
\end{conn}

\begin{example}\label{triangle}
Let $Q$ be a triangle $\bigtriangleup^2$ in $\RR^2$. The possible characteristic functions are indicated
by the following Figures \ref{egch001}.
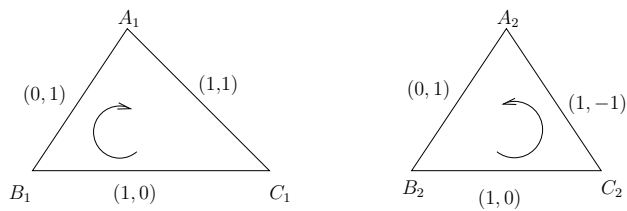
\begin{figure}[ht]
\centerline{
\scalebox{0.63}{
\input{egch001.pstex_t}
 }
 }
\caption {The characteristic functions corresponding to a triangle.}
\label{egch001}
\end{figure}
The quasitoric manifold corresponding to the first characteristic pair is $\CP^2$ with the usual $\TT^2$ action
and standard orientation, we denote it by $\CP^2_s$. The second correspond to the same $\TT^2$ action with
the reverse orientation on $\CP^2$, we denote this quasitoric manifold by $\bar{\CP^2}_s$.
\end{example}
Note that there are many non-equivariant $\TT^2$-actions on $\CP^2$. We discuss this classification in Section \ref{cobort2}. 

\begin{example}\label{square}
Suppose that $Q$ is combinatorially a square in $\RR^2$. In this case there are many possible
characteristic functions. Some examples are given by the Figure \ref{egch002}.
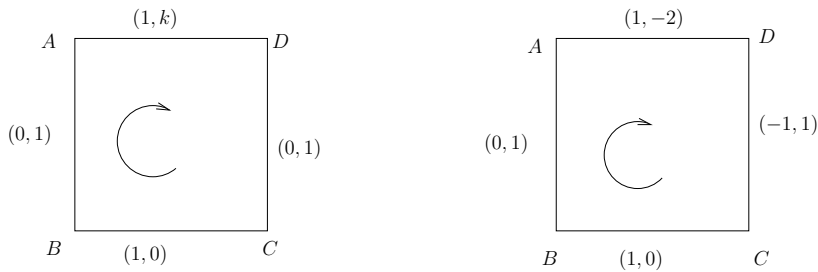
\begin{figure}[ht]
\centerline{
\scalebox{0.64}{
\input{egch002.pstex_t}
 }
 }
\caption {Some characteristic functions corresponding to a square.}
\label{egch002}
\end{figure}

The first characteristic pairs may construct an infinite family of $4$-dimensional quasitoric manifolds,
denote them by $M_k^4$ for each $k \in \ZZ$. The manifolds $\{M_k^4 : k \in \ZZ\}$ are equivariantly distinct.
Let $L(k)$ be the complex line bundle over $\CP^1$ with the first Chern class $k$. The complex manifold
$\CP(L(k) \oplus \CC)$ is the Hirzebruch surface for the integer $k$, where $\CP(\cdot)$ denotes the
projectivisation of a complex bundle. So each Hirzebruch surface is the total space of the bundle
$\CP(L(k) \oplus \CC) \to \CP^1$ with fiber $\CP^1$. It is well-known
that with the natural action of $\TT^2$ on $\CP(L(k) \oplus \CC)$ it is equivariantly homeomorphic to
$M_k^4$ for each $k$, see \cite{Oda}. That is, with respect to the $\TT^2$-action, Hirzebruch surfaces are quasitoric
manifolds where the orbit space is a combinatorial square and the corresponding characteristic map
is described in Figure \ref{egch002}.

On the other hand the second combinatorial model gives the quasitoric manifold $ ~$
$\CP^2 ~\# ~\CP^2$, the equivariant connected sum of $\CP^2$. 
\end{example}
The following remark classifies all $4$-dimensional quasitoric manifolds.
\begin{remark}\label{clasq}
Orlik and Raymond ( \cite{OR}, p. 553) show that any $4$-dimensional quasitoric manifold $M^4$ over
$2$-dimensional simple polytope is an equivariant connected sum of several copies of $\CP^2$,
$\bar{\CP^2}$ and $M_k^4$ for some $k\in \ZZ$.
\end{remark}

\section{Edge-Simple Polytopes }\label{poly}
In this section we introduce a particular type of polytope, which we call an edge-simple polytope.
This polytopes are generalization of simple polytopes. 
\begin{defn}
An $n$-dimensional convex polytope $P$ is called an $n$-dimensional edge-simple polytope if each edge
of $P$ is the intersection of exactly $(n-1)$ facets of $P$.
\end{defn}

\begin{example}
\begin{enumerate}
\item An $ n $-dimensional simple convex polytope is an $ n $-dimensional edge-simple polytope.

\item The following convex polytopes are edge-simple polytopes of dimension $3$.

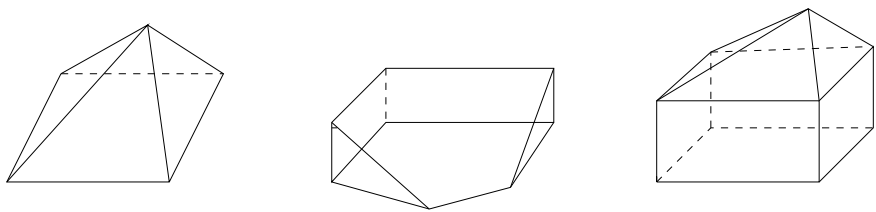
\begin{figure}[ht]
 \centerline{
 \scalebox{0.72}{
 \input{eg1.pstex_t}
  }
  }
\end{figure}

\item The dual polytope of a $ 3 $-dimensional simple convex polytope is a $ 3 $-dimensional
 edge-simple polytope. This result is not true for higher dimensional polytopes, that is if $P$ is a simple
convex polytope of dimension $n \geq 4$ the dual polytope of $P$ may not be an edge-simple polytope. For
example the dual of the $4$-dimensional standard cube in $\RR^4$ is not an edge-simple polytope.
\end{enumerate}
\end{example}

\begin{prop}
\textbf{(a)} If $ P $ is a $ 2 $-dimensional simple convex polytope then the suspension $ SP $ on $ P $
is an edge-simple polytope and $SP$ is not a simple convex polytope.

\textbf{(b)}  If $ P $ is an $ n $-dimensional simple convex polytope then the cone  $ CP $ on $ P $
is an $ (n+1) $-dimensional edge-simple polytope.
\end{prop}

\begin{proof}
{\it{\textbf{(a)}}} Let $ P $ be a $ 2 $-dimensional simple polytope with $m$ vertices
$\{v_i \colon i \in I = \{ 1, 2, \ldots, m\} \}$
and $ m $ edges $\{e_i \colon i \in I\}$. Let $ a $ and $ b $ be the other two vertices of $ SP $.
Then facets of $ SP $ are the cone $(Ce_i)_x$ on $e_i$ at $ x=a,b $. Edges of $SP$ are
$\{ xv_i \colon x=a, b$ and $ i \in I\} \cup \{e_i : i \in I\}$. The edge $ xv_i $ is the intersection of $ (Ce_{i_1})_x$
and $ (Ce_{i_2})_x$ if $ v_i = e_{i_1} \cap e_{i_2} $ for $x =a, b$ and $e_i = (Ce_i)_a \cap (Ce_i)_b$.
Hence $ SP $ is an edge-simple polytope. If $v$ is a vertex of the polytope $P$, $v$ is the intersection
of $4$ facets of $SP$. So $SP$ is not a simple convex polytope.

{\it{\textbf{(b)}}} Let $ P $ be an $ n $-dimensional simple convex polytope in $\RR^n \times 0 \subseteq \RR^{n+1}$
with $ m $ facets $\{F_i \colon i \in I = \{ 1, 2, \ldots, m \}\}$ and $ k $ vertices
$\{v_1, v_2, \ldots, v_k \}$.
Assume that the cone are taken at a fixed point $a$ in $\RR^{n+1} - \RR^n$ lying above the centroid of $P$.
Then facets of $ CP $ are $ \{ (CF_i) \colon i=1,2, \ldots, m\} \cup \{P\}$. Edges of $ CP $ are
$ \{av_i = C (\{v_i\}) \colon i=1, 2, \ldots, k\} \cup \{e_l : e_l ~\mbox{is an edge of}~ P\}$.
Since $ P $ is a simple convex polytope, each vertex $ v_i $ of $ P $ is the intersection
of exactly $ n $ facets of $ P $, namely $\{ v_i\} = \cap_{j=1}^{n}  F_{i_j} $ and each
edge $ e_l $ is the intersection of unique collection of $(n-1)$ facets $\{F_{l_1}, \ldots, F_{l_{n-1}}\}$.
Then $ C\{v_i\} = \cap_{j=1}^{n} CF_{i_j}$ and $ e_l = P \cap CF_{l_1} \cap CF_{l_2} \cap \ldots \cap CF_{l_{n-1}} $.
That is $C\{v_i\}$ and $\{e_l\}$ are the intersection of exactly $n$ facets of $ CP $.
Hence $ CP $ is an $ (n+1) $-dimensional edge-simple polytope.
\end{proof}

Cut off a neighborhood of each vertex $ v_i, i= 1, 2, \ldots, k $ of an $n$-dimensional edge-simple polytope
$ P \subset \RR^n $ by an affine hyperplane $ H_i, i= 1, 2, \ldots, k $ in $\RR^n$ such that $H_i \cap H_j \cap P$
are empty sets for $i \neq j$.
Then the remaining subset of the convex polytope $P$ is a simple convex polytope of dimension $n$, denote it by $ Q_P $.
Suppose $ P_{ H_i} = P \cap H_i = H_i \cap Q_P$ for $ i=1,2, \ldots k $. Then $ P_{H_i}$ is a facet of $Q_P$
called the facet corresponding to the vertex $v_i$ for each $i= 1, \ldots, k $. Since each vertex of $P_{H_i}$
is an interior point of an edge of $P$ and $P$ is an edge-simple polytope, $P_{H_i}$ is an $(n-1)$-dimensional simple convex
polytope for each $i= 1, 2, \ldots, k $.

\begin{lemma}
Let $ F $ be a codimension $l < n$ face of $ P $. Then $F$ is the intersection of unique set of $ l$ facets of $P$.
\end{lemma}

\begin{proof}
The intersection $ F \cap Q_P $ is a codimension $l$ face of $ Q_P $
not contained in $\cup_{i=0}^{k} \{P_{H_i}\}$.
Since $ Q_P $ is a simple convex polytope, $ F \cap Q_P = \cap_{j=1}^{l} F_{i_j}^{\prime} $ for some facets
$ \{F_{i_1}^{\prime}, \ldots, F_{i_l}^{\prime}\}$ of $ Q_P $. Let $ F_{i_j} $ be the unique facet of $ P $
such that $F_{i_j}^{\prime} \subseteq F_{i_j} $. Then $ F = \cap_{1}^{l} F_{i_j} $.
Hence each face of $P$ of codimension $l < n$ is the intersection of unique set of $ l$ facets of $P$.
\end{proof}

\begin{remark}
If $v_i$ is the intersection of facets $ \{F_{i_1}, \ldots, F_{i_l}\}$ of $P$ for some positive
integer $l$, the facets of $P_{H_i}$ are $ \{P_{H_i} \cap F_{i_1}, \ldots, P_{H_i} \cap F_{i_l}\}$.
\end{remark}

\section{Construction of Manifolds with Boundary}\label{def}
Let $ P $ be an edge-simple polytope of dimension $ n $ with $ m $ facets $ F_1, \ldots, F_{m}$
and $k$ vertices $ v_1, \ldots, v_{k} $.
Let $ e $ be an edge of $ P $. Then $ e $ is the intersection of unique collection of $(n-1)$ facets
$ \{F_{i_j} : j= 1, \ldots, (n-1)\}$. Let $ \mathcal{F}(P) = \{ F_1, \ldots, F_{m}\}$ and $\mathbb{F}_2^{n-1}$
be the $n-1$ dimensional vector space over $\mathbb{F}_2$, the field of integer modulo $2$.

\begin{defn}
The functions $ \lambda \colon \mathcal{F}(P) \to \ZZ^{n-1}/\ZZ_2 $ and $ \lambda^s \colon \mathcal{F}(P) \to \mathbb{F}_2^{n-1} $
are called the isotropy function and $\mathbb{F}_2$-isotropy function respectively of the edge-simple polytope $P$
if the set of vectors $ \{\lambda(F_{i_1}), \ldots, \lambda(F_{i_{n-1}})\}$ and
$ \{\lambda^s(F_{i_1}), \ldots, \lambda^s(F_{i_{n-1}})\}$ form a basis of $ \ZZ^{n-1} $ and $\mathbb{F}_2^{n-1} $ respectively whenever
the intersection of the facets $ \{F_{i_1}, \ldots, F_{i_{n-1}}\}$ is an edge of $P$.

The vectors $ \lambda_i := \lambda(F_i) $ and $ \lambda^s_i := \lambda^s(F_i) $ are called isotropy vectors
and $\mathbb{F}_2$-isotropy vectors respectively.
\end{defn}
We define some isotropy functions of the edge-simple polytopes $I^3$ and $P_0$ in examples \ref{egchar1} and \ref{egchar2} respectively.
\begin{remark}
It may not be possible to define an isotropy function on the set of facets of all edge-simple polytopes. For example there
does not exist an isotropy function of the standard $n$-simplex $\bigtriangleup^n$ for each $n\geq 3$.
\end{remark}
\subsection{Manifolds with Quasitoric Boundary}\label{defqb}
Let $ F $ be a face of $ P $ of codimension $ l < n$.
Then $ F $ is the intersection of a unique collection of $ l $ facets $F_{i_1}, F_{i_2}, \ldots, F_{i_l}$ of $ P $.
Let $ \TT_F $ be the torus subgroup of $ \TT^{n-1}$ corresponding to the submodule generated by
$ \lambda_{i_1}, \lambda_{i_2}, \ldots, \lambda_{i_l} $ in $ \ZZ^{n-1} $. Assume $\TT_v = \TT^{n-1}$ for each vertex $v$ of $P$.
We define an equivalence relation $\sim$ on the product $ \TT^{n-1}\times P $ as follows.
\begin{equation}
 (t,p) \sim (u,q) ~ \mbox{if and only if} ~  p = q ~ \mbox{ and} ~ tu^{-1} \in \TT_{F}
\end{equation}
where $ F \subset P$ is the unique face containing $ p $ in its relative interior. We denote the quotient
space $ (\TT^{n-1} \times P)/ \sim $ by $ X(P, \lambda) $. The space $ X(P, \lambda) $ is not a manifold except
when $P$ is a $2$-dimensional polytope.
If $P$ is $2$-dimensional polytope the space $ X(P, \lambda) $ is homeomorphic to the $3$-dimensional sphere.

But whenever $n > 2$ we can construct a manifold with boundary from the space $ X(P, \lambda) $.
We restrict the equivalence relation $\sim$ on the product $(\TT^{n-1} \times Q_P)$ where $Q_P \subset P$ is a simple
polytope as constructed in Section \ref{poly} corresponding to the edge-simple polytope $P$.
Let $ W(Q_P, \lambda) = (\TT^{n-1} \times Q_P )/\sim ~ \subset X(P, \lambda)$ be the quotient space.
The natural action of $ \TT^{n-1} $ on $ W(Q_P, \lambda) $ is induced by the group operation in $ \TT^{n-1} $.

\begin{theorem}
The space $ W(Q_P, \lambda) $ is a manifold with boundary. The boundary is a disjoint union of quasitoric manifolds.
\end{theorem}
For each edge $ e $ of $ P $, $ e^{\prime} = e \cap Q_P $ is an edge of the simple convex polytope $ Q_P $.
Let $ U_{e^{\prime}}$ be the open subset of $ Q_P $ obtained by deleting all facets of $ Q_P $ that does not
contain $ e^{\prime} $ as an edge. Then the set $ U_{e^{\prime}}$ is diffeomorphic to
$ I^{0} \times \RR_ {\geqslant 0}^{n-1} $ where $I^0$ is the open interval $(0,1)$ in $\RR$. The facets of
$ I^{0} \times \RR_ {\geqslant 0} ^{n-1} $ are $ I^{0} \times \{ x_1 = 0 \}, \ldots, I^{0} \times \{ x_{n-1} = 0 \} $
where $ \{x_j =0, ~ j=1, 2, \ldots, {n-1} \}$ are the coordinate hyperplanes in $ \RR^{n-1} $.
Let $ F_{i_1}^{\prime}, \ldots, F_{i_{n-1}}^{\prime} $ be the facets of $ Q_P $ such that
$ \cap_{j=1}^{n-1} F_{i_j}^{\prime} = e^{\prime} $.
Suppose the diffeomorphism $ \phi \colon U_{e^{\prime}} \to I^{0} \times \RR_ {\geqslant 0} ^{n-1} $
sends $ F_{i_j}^{\prime} \cap U_{e^{\prime}} $ to $ I^0 \times \{ x_j = 0 \} $ for all $j = 1, 2, \ldots, n-1$. Define an
isotropy function $ \lambda _{e} $ on the set of all facets of  $ I^{0} \times \RR_ {\geqslant 0} ^{n-1} $
by $ \lambda_{e} ( I^0 \times \{ x_j = 0 \} ) = \lambda_{i_j}$ for all $ j = 1, 2, \ldots, {n-1} $.
We define an equivalence relation $ \sim _e $ on $ ( \TT^{n-1} \times I^0 \times \RR_ {\geqslant 0}^{n-1} ) $
as follows.
\begin{equation}
 (t,b,x) \sim _e (u,c,y) ~ \mbox{ if and only if} ~ (b,x) = (c,y) ~ \mbox{ and} ~ tu^{-1} \in \TT_{\phi{(F)}}.
\end{equation}
where $ \phi{(F)} $ is the unique face of $ I^0 \times \RR_ {\geqslant 0} ^{n-1}$ containing $ (b,x) $ in its
relative interior, for a unique face $F$ of $U_{e^{\prime}} $ and $\TT_{\phi{(F)}} = \TT_F $.
So for each $ a \in I^0 $ the restriction of $ \lambda _{e} $ on
$\{ ( \{ a \} \times \{ x_j = 0 \}) : j = 1, 2, \ldots, {n-1}\} $ defines a characteristic function
(see Definition \ref{charfun}) on the set of facets of $ \{ a \} \times \RR_ {\geqslant 0} ^{n-1} $.
From the constructive definition of quasitoric manifold given in \cite{DJ} it is clear that the quotient space
$ \{ a \} \times ( \TT^{n-1} \times \RR_ {\geqslant 0} ^{n-1} ) / \sim_e $ is diffeomorphic to $ \{a\} \times \RR^{2(n-1)} $.
Hence the quotient space
$$ (\TT^{n-1} \times I^0 \times \RR_ {\geqslant 0} ^{n-1})/\sim_e ~ = ~ I^0 \times ( \TT^{n-1} \times \RR_ {\geqslant 0} ^{n-1} ) / \sim_e ~
\cong ~ I^0 \times \RR^{2(n-1)}.$$
Since the maps $ \pi : ( \TT^{n-1} \times U_{ e^{\prime}}) \to ( \TT^{n-1} \times U_{ e^{\prime} })/ \sim $
and $ \pi_e : (\TT^{n-1} \times I^0 \times \RR_{\geqslant 0}^{n-1} ) \to (\TT^{n-1} \times I^0 \times \RR_{\geqslant 0}^{n-1})/\sim_e $
are quotient maps and $ \phi $ is a diffeomorphism, the following commutative diagram ensure that the lower horizontal map
$ \phi _e $ is a homeomorphism.

\begin{equation}
\begin{CD}
( \TT^{n-1} \times U_{ e^{\prime}}) @>id \times \phi>> (\TT^{n-1} \times I^0 \times \RR_{\geqslant 0}^{n-1}) \\
@V\pi VV  @V \pi_e VV @. \\
( \TT^{n-1} \times U_{e^{\prime}}) / \sim @>\phi_e >> (\TT^{n-1} \times I^0 \times \RR_{\geqslant 0}^{n-1})/\sim_e @>\cong>> I^0 \times
\RR^{2(n-1)}
\end{CD}
\end{equation}

Let $v_1^{\prime}$ and $v_2^{\prime}$ be the vertices of the edge $e^{\prime}$ of $Q_P$.
Suppose $ H_{1} \cap e^{\prime} =\{ v_{1}^{\prime} \} $ and  $ H_{2} \cap e^{\prime} =\{ v_{2}^{\prime} \} $, where $ H_{1} $
and $ H_{2} $ are affine hyperplanes as considered in Section \ref{poly} corresponding to the vertices $v_1$ and $v_2$ of $e$ respectively.
Let $ U_{ v_1^{\prime} } $ and $ U_{ v_2^{\prime} } $ be the open subset of $ Q_P $ obtained by deleting all facets of $ Q_P $
not containing $ v_{1}^{\prime} $ and $ v_{2}^{\prime} $ respectively.
Hence there exist diffeomorphism $ \phi^1 : U_{ v_1^{\prime} } \to [0,1) \times \RR_{\geqslant 0}^{n-1} $ and
$ \phi^2 : U_{ v_2^{\prime} } \to [0,1) \times \RR_{\geqslant 0}^{n-1}$ satisfying the same property as the map $\phi$.
We get the following commutative diagram and homeomorphisms $ \phi_e^j $ for $ j = 1, 2 $.
\begin{equation}
\begin{CD}
( \TT^{n-1} \times U_{ v_j^{\prime}}) @>id \times \phi^j>> (\TT^{n-1} \times [0,1) \times \RR_{\geqslant 0}^{n-1})\\
@V\pi VV  @V \pi_e VV @.\\
(\TT^{n-1} \times U_{v_j^{\prime}})/\sim @>\phi_e^j>> (\TT^{n-1} \times [0,1) \times \RR_{\geqslant 0}^{n-1})/\sim_e @>\cong>> [0,1) \times \RR^{2(n-1)}
\end{CD}
\end{equation}

Hence each point of $ (\TT^{n-1} \times Q_P)/ \sim $ has a neighborhood homeomorphic to an open subset of $ [0,1) \times \RR^{2(n-1)} $.
So $ W( Q_P, \lambda ) $ is a manifold with boundary.
From the above discussion the interior of $ W( Q_P, \lambda ) $ is
$$\cup_{_{e^{\prime}}} (\TT^{n-1} \times U_{e^{\prime}})/\sim ~ = ~ W(Q_P, \lambda) \smallsetminus \{(\TT^{n-1} \times
\sqcup_{i=1}^{k} P_{H_i})/\sim \}$$
and the boundary is $\sqcup_{i=1}^{k} \{(\TT^{n-1} \times P_{H_i})/\sim \}$.
Let $F(H)_{i_j}$ be a facet of $P_{H_i}$. So there exists a unique facet $F_j$ of $P$ such that  $F(H)_{i_j} = F_j \cap Q_P \cap H_i$.
The restriction of the function $\lambda$ on the set of all facets of $P_{H_i}$ (namely $\lambda(F(H)_{i_j}) = \lambda_j$) give a
characteristic function of a quasitoric manifold over $P_{H_i}$. Hence restricting the equivalence relation $ \sim$ on
$(\TT^{n-1} \times P_{H_i})$ we get that the quotient space $W_i = (\TT^{n-1} \times P_{H_i})/\sim$ is a quasitoric
manifold over $P_{H_i}$.
Hence the boundary $ \partial{W}(Q_P, \lambda) $ is the disjoint union $ \sqcup_{_{i=1}}^{^{k}} W_i$, where $W_i $ is a quasitoric manifold.
So $ W( Q_P, \lambda ) $ is a manifold with quasitoric boundary.

In Section \ref{ori} we have shown that these manifolds with quasitoric boundary are orientable.

\begin{example}\label{egchar1}
An isotropy function of the standard cube $I^3$ is described in the following Figure \ref{egc1}.
Here simple convex polytopes $P_{H_1}, \ldots, P_{H_8}$ are triangles. The restriction of the isotropy
function on $P_{H_i}$ gives that the space $(\TT^2 \times P_{H_i})/\sim$ is the complex projective space
either $\CP^2$ or $\bar{\CP^2}$. Since antipodal map in $\RR^3$ is an orientation reversing map we can show that
the disjoint union $\sqcup_{i=1}^4 \CP^2 \sqcup_{i=1}^4 \bar{\CP^2}$ is the boundary of $(\TT^2 \times Q_{I^3})/ \sim$.

\begin{figure}[ht]
        \centerline{
           \scalebox{0.68}{
            \input{egc1.pstex_t}
            }
          }
       \caption {An isotropy function $\lambda$ of the edge-simple polytope $I^3$}
        \label{egc1}
      \end{figure}
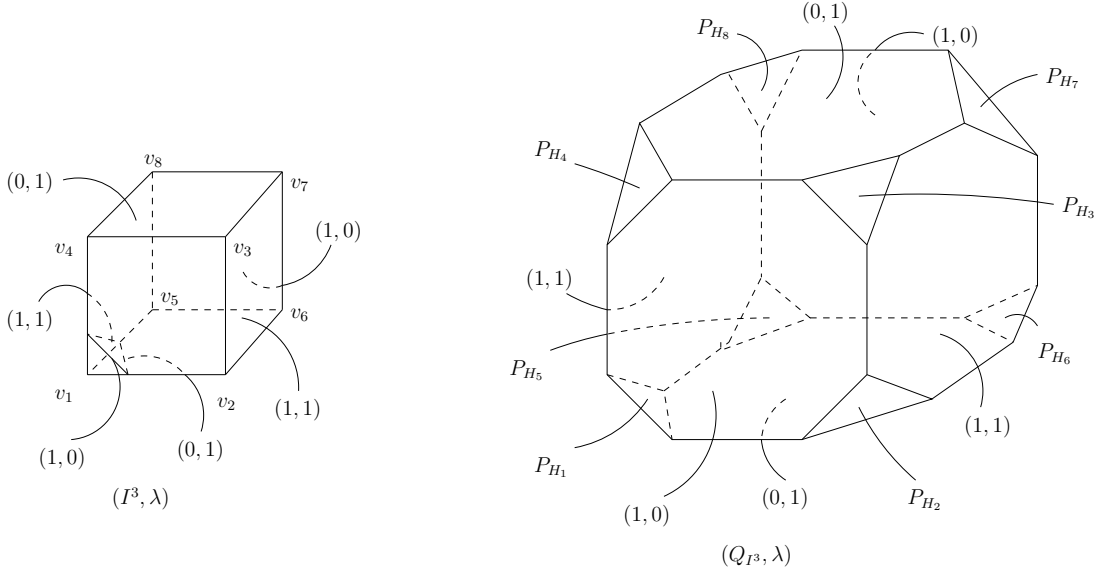
\end{example}

\begin{example}\label{egchar2}
In the following Figure \ref{egc2} we define an isotropy function of the edge-simple polytope $ P_0$.
Here simple convex polytopes $P_{H_{1}},P_{H_{2}}, P_{H_{3}}, P_{H_{4}}$ are triangles and the simple
convex polytope $P_{H_{5}}$ is a rectangle. The restriction of the isotropy function on $P_{H_{i}}$
gives that the space $(\TT^2 \times P_{H_{i}})/\sim$ is either $\CP^2$ or $\bar{\CP^2}$
for each $i\in \{1,2,3,4\}$ and $(\TT^2 \times P_{H_{5}})/\sim$ is $\CP^1 \times \CP^1$. Hence the space
$\sqcup_{i=1}^2 \CP^2 \sqcup_{i=1}^2 \bar{\CP^2} \sqcup (\CP^1 \times \CP^1)$ is the boundary of
$W(Q_{P_0}, \lambda) := (\TT^2 \times Q_{P_0})/ \sim$, see subsection \ref{cobort2}.

\begin{figure}[ht]
        \centerline{
           \scalebox{0.73}{
            \input{egc2.pstex_t}
            }
          }
       \caption {An isotropy function $\lambda$ of the edge-simple polytope $P_0$}
        \label{egc2}
      \end{figure}
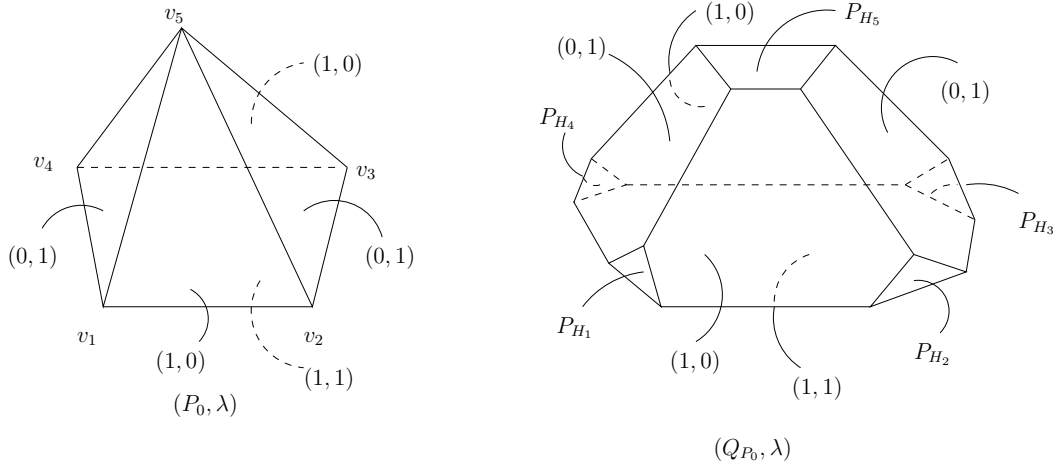
\end{example}

\subsection{Manifolds with small cover boundary}\label{smbd}
We assign each face $F$ to the subgroup $G_F$ of $\mathbb{F}_2^{n-1}$ determined by the vectors
$\lambda^s_{i_1}, \ldots, \lambda^s_{i_l}$
where $F$ is the intersection of the facets $F_{i_1}, \ldots, F_{i_l}$. Let $\sim_s$ be an equivalence relation
on $(\mathbb{F}^{n-1}_2 \times P)$ defined by the following.
\begin{equation}
 (t,p) \sim_s (u,q) ~ \mbox{if and only if} ~  p = q ~ \mbox{ and} ~ t-u \in G_{F}
\end{equation}
where $ F \subset P$ is the unique face containing $ p $ in its relative interior.
The quotient space $(\mathbb{F}^{n-1}_2 \times Q_P)/ \sim_s ~ \subset (\mathbb{F}^{n-1}_2 \times P)/ \sim_s$,
denoted by $S(Q_P, \lambda^s)$, is a manifold with boundary.
This can be shown by the same arguments given in the subsection \ref{defqb}. The boundary of this manifold is
$\{(\mathbb{F}_2^{n-1} \times \sqcup_{i=1}^{k} P_{H_i})/\sim_s \} = \sqcup_{i=1}^{k}\{(\mathbb{F}_2^{n-1} \times P_{H_i})/\sim_s \}$.
Clearly the restriction of the $\mathbb{F}_2$-isotropy function $ \lambda^s $ on the set of all facets of $P_{H_i}$ gives the
characteristic function of a small cover over $P_{H_i}$. So $(\mathbb{F}_2^{n-1} \times P_{H_i})/\sim_s$ is a small cover for each
$i=0, \ldots, k $. Hence $S(Q_P, \lambda^s)$ is a manifold with small cover boundary.

\subsection{Some observations}\label{quot}
The set of all facets of the simple convex polytope $ Q_P $ are
$ \mathcal{F}({Q_P}) = \{ P_{ H_j } \colon j= 1,2, \ldots, k \} \cup \{ F_i ^{\prime} \colon i= 1,2, \ldots, m \} $,
where $ F_i^{\prime} = F_i \cap Q_P $ for a unique facets $F_i$ of $ P $.
We define the function $ \eta \colon \mathcal{F}( Q_P ) \to \ZZ^n/\ZZ_2 $ as follows.
\begin{equation}
 \eta( F ) = \left\{ \begin{array}{ll} [( 0, \ldots, 0, 1 )] \in \ZZ^n/\ZZ_2 & \mbox{if} ~ F = P_{H_j}~ \mbox{and} ~ j~ \in \{1, \ldots, k\} \\
 {[{\lambda_i}, 0]} \in \ZZ^{n-1}/\ZZ_2 \times\{0\} \subset \ZZ^n/\ZZ_2 & \mbox{if}~ F = F_i ~ \mbox{and} ~ i \in \{ 1, 2, \ldots, m\}
\end{array} \right.
\end{equation}
So the function $ \eta $ satisfies the condition for the characteristic function (see Definition \ref{charfun}) of a quasitoric manifold over
the $n$-dimensional simple convex polytope $Q_P$.
Hence from the characteristic pair $(Q_P, \eta)$ we can construct the quasitoric manifold $ M(Q_P, \eta ) $
over $Q_P$. There is a natural $ \TT^n $ action on $M(Q_P, \eta )$.
Let $ \TT_H $ be the circle subgroup of $ \TT^n $ determined by the submodule
$ \{0\} \times \{0\} \times \ldots \times \{0\} \times \ZZ $ of $ \ZZ^n $.
Hence $ W( Q_P, \lambda ) $ is the orbit space of the circle $ \TT_H $ action on $ M( Q_P, \eta ) $.
The quotient map $ \phi_H \colon M( Q_P, \eta ) \to W( Q_P, \lambda) $ is not a fiber bundle map.

\begin{remark}
The manifold $S(Q_p, \lambda_s)$ with small cover boundary constructed in subsection \ref{smbd} is the orbit space of $\ZZ_2$ action
on a small cover.
\end{remark}

\section{Orientability of $ W( Q_P, \lambda )$}\label{ori}
Suppose $W = W(Q_P, \lambda)$. The boundary $ \partial{W} $ has a collar neighborhood in $W$.
Hence by the proposition $2.22$ of \cite{Hat} we get $H_i(W,\partial W) = \widetilde{H}_i(W/\partial W)$ for all $i$.
We show the space $W/\partial{W}$ has a $CW$-structure. Actually we show that corresponding to each edge of $P$ there
exist an odd-dimensional cell of $W/\partial{W}$.
Realize $Q_P$ as a simple convex polytope in $\RR^n$ and choose a linear functional $\phi: \RR^n \to \RR$ which distinguishes
the vertices of $Q_P$, as in the proof of Theorem $3.1$ in \cite{DJ}. The vertices are linearly ordered according to
ascending value of $\phi$. We make the $1$-skeleton of $Q_P$ into a directed graph by orienting each edge such that
$\phi$ increases along edges. For each vertex $v$ of $ Q_P $ define its index, $ ind(v) $, as the number of incident
edges that point towards $ v $. Suppose $\mathcal{V}(Q_P)$ is the set of all vertices and $\mathcal{E}(Q_P)$ is the set
of edges of $Q_P$. For each $j \in \{ 1, 2, \ldots, n\} $, let

$$I_j = \{(v,e_v) \in \mathcal{V}(Q_P) \times \mathcal{E}(Q_P) : ind(v)=j ~ and~  e_v ~ is ~the~ incident ~edge~ that ~ points$$
$$ towards~  v~ such~ that~  e_v = e \cap Q_P ~ for ~ an ~ edge ~  e ~of ~ P\}.$$
Suppose $(v, e_v) \in I_j$. Let $F_{e_v} \subset Q_P$ denote the smallest face which contains the inward pointing
edges incident to $v$. Then $F_{e_v}$ is a unique face not contained in any $P_{H_i}$.
Let $U_{e_v}$ be the open subset of $F_{e_v}$ obtained by deleting all faces of $F_{e_v}$ not containing the edge $e_v$.
The restriction of the equivalence relation $\sim$ on $ (\TT^{n-1}\times U_{e_v})$ gives that the quotient space
$ (\TT^{n-1}\times U_{e_v})/\sim $
is homeomorphic to the open disk $B^{2j-1}$. Hence the quotient space $(W/\partial W)$ has a $CW$-complex structure with
odd dimensional cells and one zero dimensional cell only. The number of $(2j-1)$-dimensional cell is $|I_j|$, the cardinality of
$I_j$ for $j=1, 2, \ldots, n$. So we get the following theorem.

\begin{theorem}
$
H_i(W, \partial W) = \left\{ \begin{array}{ll} \displaystyle \bigoplus_{|I_j|} \ZZ & \mbox{if} ~ i = 2j-1~
\mbox{and} ~ j \in \{1, \ldots, n\} \\
 \ZZ & \mbox{if}~ i = 0 \\
 0 & \mbox{otherwise}
\end{array} \right.
$
\end{theorem}
When $j=n$ the cardinality of $I_j$ is one. So $H_{2n-1}(W, \partial W) = \ZZ$. Hence  $W$ is an orientable
manifold with boundary.

\begin{example} We adhere the notations of Example \ref{egchar2}. Observe that $I_3 = \{(v_{14}, e_{v_{14}})\}$,
$I_2 = \{(v_8, e_{v_8}), (v_{13}, e_{v_{13}}), (v_{15}, e_{v_{15}})\}$ and $I_1 = \{(v_3, e_{v_3}), (v_6, e_{v_6}), (v_9, e_{v_9})\}$.
The face $F_{e_{v_{13}}}$ corresponding to the point $ (v_{13}, e_{v_{13}})$ is $v_0v_3v_5v_{13}v_{12}v_1$.
Thus we can give a $CW$-structure of $W(Q_{P_0}, \lambda)/ \partial{W(Q_{P_0}, \lambda)}$ with one $0$-cell,
two $1$-cells, three $3$-cells and one $5$-cell.
 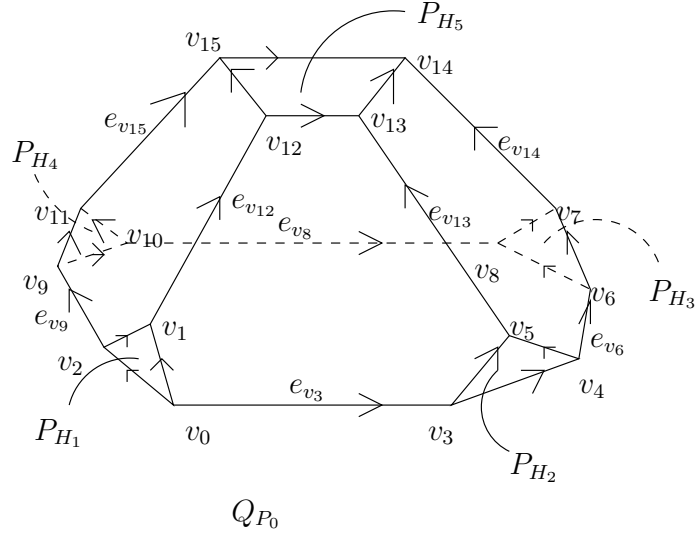
\begin{figure}[ht]
\centerline{
\scalebox{.97}{
\input{egoricob.pstex_t}
 }
 }
\caption {The index function of $Q_{P_0}$.}
\label{egoricob}
\end{figure}
\end{example}

In \cite{DJ} the authors showed that the odd dimensional homology of quasitoric
manifolds are zero. So $H_{2i-1} (\partial W) = 0$ for all $i$. Hence we get the following exact sequences for
the collared pair $(W, \partial{W})$.

\begin{equation}
\begin{CD}
0 \to H_{2n-1}(W) @>j_{\ast}>> H_{2n-1}(W, \partial W)@>\partial>> H_{2n-2}(\partial{W}) @>i_{\ast}>> H_{2n-2}(W)\to 0 \\
\vdots @. \vdots @. \vdots @. \vdots \\
0 \to H_{3}(W) @>j_{\ast}>> H_{3}(W, \partial W)@>\partial>> H_{2}(\partial{W}) @>i_{\ast}>> H_{2}(W)\to 0 \\
0 \to H_{1}(W) @>j_{\ast}>> H_{1}(W, \partial W)@>\partial>> H_{0}(\partial{W}) @>i_{\ast}>> H_{0}(W) \twoheadrightarrow \ZZ
\end{CD}
\end{equation}

Where $\ZZ \cong H_{0}(W, \partial W)$. Let $ (h_{i_0},\ldots, h_{i_{n-1}})$ be the $h$-$vector$ of $P_{H_i}$, for $i = 1, 2, \ldots, k$.
The definition of $h$-$vector$ of simple convex polytope is given in \cite{DJ}.
Hence the Euler characteristic of the manifold $W$ with quasitoric boundary is
$\Sigma_{i=1}^k \Sigma_{j=0}^{n-1} {h_{i_j}} - \Sigma_{j=1}^{n-1} {|I_j|}$.

Fix the standard orientation on $\TT^{n-1}$. Let $I_n = \{(v^{\prime}, e_{v^{\prime}})\}$. Then the $(2n-1)$-dimensional cell
$ (\TT^{n-1}\times U_{e_{v^{\prime}}})/\sim \subset W$ represents a fundamental class of $W/ \partial{W}$ with coefficient in $\ZZ$.
Thus an orientation of $U_{e_{v^{\prime}}}$ (hence of $Q_P$) determines an orientation of $W$. Note that an orientation
of $Q_P$ is induced by orienting the ambient space $\RR^n$.

So the boundary orientation on $P_{H_i}$ induced from the
orientation of $Q_P$ gives the orientation on the quasitoric manifold $W_i \subset \partial{W}$. In the next section
we consider the orientation of $Q$'s and $Q_P$'s induced from the standard orientation of their ambient spaces.

\section{Torus Cobordism of Quasitoric Manifolds}\label{cobort2}
Let $\mathfrak{C}$ be the following category: the objects are all quasitoric manifolds and morphisms are
torus equivariant maps between quasitoric manifolds. We are considering torus cobordism in this category only.

\begin{defn}
Two $2n$-dimensional quasitoric manifolds $M_1$ and $M_2$ are said to be $\mathbb{T}^n$-cobordant if there exists
an oriented $\mathbb{T}^n$ manifold $W$ with boundary $\partial W $ such that $\partial W $ is $\mathbb{T}^n$
equivariantly diffeomorphic to $ M_1 \sqcup (-M_2)$ under an
orientation preserving diffeomorphism. Here $-M_2$ represents the reverse orientation of $M_2$.
\end{defn}
We denote the $\TT^n$-cobordism class of quasitoric $2n$-manifold $M$ by $[M]$.
\begin{defn}
 The $n$-th torus cobordism group is the group of all cobordism classes of $2n$-dimensional quasitoric manifolds
with the operation of disjoint union. We denote this group by $CG_n$.
\end{defn}

Let $M \to Q$ be a $4$-dimensional quasitoric manifold over the square $Q$ with
the characteristic function $\eta : \mathcal{F}(Q) \to \ZZ^2/\ZZ_2$.
We construct an oriented $\TT^2$ manifold $W$ with boundary $\partial W $, where $\partial W $ is equivariantly
homeomorphic to either $ -M \sqcup \sqcup_{k_1} \CP^2 \sqcup \sqcup_{k_2} \bar{\CP^2}$ or
$ M \sqcup \sqcup_{k_1} \CP^2 \sqcup \sqcup_{k_2} \bar{\CP^2}$ for some integer $k_1, k_2$.
To show this we construct a $3$-dimensional edge-simple polytope $P_{\mathcal{E}}$ such that
$P_{\mathcal{E}}$ has exactly one vertex $O$ which is the intersection of $4$ facets with
$P_{\mathcal{E}} \cap H_O = Q $ and other vertices of $P_{\mathcal{E}}$ are intersection of $3$ facets.
We define an isotropy function $\lambda$, extending the characteristic function $\eta$ of $M$,
from the set of facets of $P_{\mathcal{E}}$ to $\ZZ^2/\ZZ_2$.
Then $W(Q_{P_{\mathcal{E}}}, \lambda)$ is the required oriented $\TT^2$ manifold with quasitoric boundary.
We have done an explicit calculation in the following.

Let $Q = ABCD$ be a rectangle (see Figure \ref{eg2}) belongs to $\{(x,y,z) \in \RR^3_{\geq 0} : x+y+z=1\}$.
Let $\eta : \{ AB, BC, CD, DA \} \to \ZZ^2/\ZZ_2 $ be the characteristic function for a quasitoric
manifold $M$ over $ABCD$ such that the characteristic vectors are

$$ \eta(AB)= \eta_1, ~ \eta(BC)= \eta_2, ~ \eta(CD)= \eta_3 ~ \mbox{and} ~ \eta(DA)= \eta_4.$$

We may assume that $\eta_1 =(0,1)$ and $ \eta_2=(1,0)$.
From the classification results given in subsection \ref{defegm}, it is enough to consider the following cases only.

\begin{equation}\label{ca1}
 \eta_3 =(0,1) ~ \mbox{ and} ~ \eta_4 = (1,0)
\end{equation}

\begin{equation}\label{ca2}
\eta_3 =(0,1) ~ \mbox{ and} ~ \eta_4 = (1,k), ~ k= 1 ~\mbox{or} ~ -1
\end{equation}

\begin{equation}\label{ca3}
\eta_3 =(0,1) ~ \mbox{ and} ~ \eta_4 = (1,k), ~ k \in \ZZ -\{-1, 0, 1\}
\end{equation}

\begin{equation}\label{ca4}
\eta_3 =(-1,1) ~ \mbox{ and} ~ \eta_4 = (1,-2)
\end{equation}

\textbf{For the case} \ref{ca1}: In this case the edge-simple polytope $\widetilde{P}_1$, given in Figure \ref{eg2}, is the required
edge-simple polytope. The isotropy vectors of $\widetilde{P}_1$ are given by
\[
\lambda(OGH) = \eta_1, ~ \lambda(OHI) = \eta_2, ~ \lambda(OIJ) = \eta_3,
~ \lambda(OGJ) = \eta_4 ~ \mbox{and} ~ \lambda(GHIJ) = \eta_1 + \eta_2.
\]
So we get an oriented $\TT^2$ manifold $W(Q_{\widetilde{P}_1}, \lambda)$ with quasitoric boundary where the boundary is the quasitoric
manifold $-M \sqcup \sqcup_{k_1} \CP^2 \sqcup \sqcup_{k_2} \bar{\CP^2}$ for some integers $k_1, k_2$.
Note that orientation on $\widetilde{P}_1 \subset \RR^3_{\ge 0}$ comes from the standard orientation of $\RR^3$.
Let $A^{\prime}$ and $B^{\prime}$ be the midpoints of $GJ$ and $HI$ respectively. Let $\mathcal{H}$ be the plane
passing through $O, A^{\prime}$ and $B^{\prime}$ in $\RR^3$.
Since a reflection in $\RR^3 $ is an orientation reversing homeomorphism, it is easy to observe that the reflection on $\mathcal{H}$
induces the following orientation reversing equivariant homeomorphisms.
\begin{equation}\label{hiz1}
 (\TT^2 \times \widetilde{P}_{1_I})/\sim \to (\TT^2 \times \widetilde{P}_{1_H})/\sim~
\mbox{and} ~(\TT^2 \times \widetilde{P}_{1_J})/\sim \to (\TT^2 \times \widetilde{P}_{1_G})/\sim.
\end{equation}
So $k_1 = k_2$. Since $[\bar{\CP^2}] = -[\CP^2]$, $ [M] = 0 [\CP^2]$. Identifying the corresponding boundaries
of $W(Q_{\widetilde{P}_1}, \lambda)$ via the equivariant homeomorphisms of equation \ref{hiz1} we get that $M$ is
the boundary of a nice oriented $\TT^2$ manifold. By 'nice manifold' we mean it has good CW-complex structures.

\textbf{For the case } \ref{ca2}: In this case $\lvert det(\eta_2,\eta_4) \rvert =1 $.
Let $O$ be the origin of $\RR^3$. Let $C_Q$ be the open cone on rectangle $ABCD$ at the origin $O$.
Let $G, H, I, J$ be points on extended $OA, OB, OC, OD$ respectively. Let $E$ and $F$ be two points in the interior of the open cones on
$AB$ and $CD$ at $O$ respectively such that $\lvert OG\rvert < \lvert OE \rvert$, $\lvert OH \rvert < \lvert OE \rvert$ and
$\lvert OI\rvert < \lvert OF \rvert$, $\lvert OJ \rvert < \lvert OF \rvert$. May assume that $OH = OI$, $OG=OJ$, $HE=EG$ and $IF = FJ$.
Then the convex polytope $P_1 \subset C_Q$ on the set of vertices $\{ O, G, E, H, I, F, J \}$ is an edge-simple polytope
(see Figure \ref{eg2}) of dimension $3$. Define a function, denote by $\lambda$, on the set of facets of $P_1$ by

\begin{equation}
 \begin{array}{ll}  \lambda(OGEH) = \eta_1, ~ \lambda(OHI) = \eta_2, ~ \lambda(OJFI) = \eta_3, ~ \lambda(OJG) = \eta_4,\\
 \lambda(HIFE) = \eta_4 ~ \mbox{and} ~ \lambda(GJFE) = \eta_2.
\end{array}
\end{equation}

Hence $\lambda$ is an isotropy function on the edge-simple polytope $P_1$. The boundary of the oriented
$\TT^2$ manifold $ W(Q_{P_1}, \lambda)$ is the quasitoric manifold $-M \sqcup \sqcup_{k_1}\CP^2 \sqcup \sqcup_{k_2} \bar{\CP^2}$
for some integers $k_1, k_2$. Similarly to the previous case we can show that suitable reflections
induce the following orientation reversing equivariant homeomorphisms.
\begin{equation}\label{hiz2}
 \begin{array}{ll} (\TT^2 \times P_{1_H})/\sim \to (\TT^2 \times P_{1_I})/\sim,
~ (\TT^2 \times P_{1_E})/\sim \to (\TT^2 \times P_{1_F})/\sim\\\\
\mbox{and}~ (\TT^2 \times P_{1_G})/\sim \to (\TT^2 \times P_{1_J})/\sim.
\end{array}
\end{equation}
So $k_1 = k_2$. Hence $ [M] = 0 [\CP^2]$. Identifying the corresponding boundaries of $ W(Q_{P_1}, \lambda)$
via the equivariant homeomorphisms of equation \ref{hiz2} we get that $M$ is
the boundary of a nice oriented $\TT^2$ manifold.

\begin{figure}[ht]
\centerline{
\scalebox{0.58}{
\input{eg2.pstex_t}
 }
\scalebox{0.58}{
\input{eg3.pstex_t}
 }
 }
\caption {The edge-simple polytope $P_1, \widetilde{P}_1$ and the convex polytope $P_1^{\prime}$ respectively.}
\label{eg2}
\end{figure}
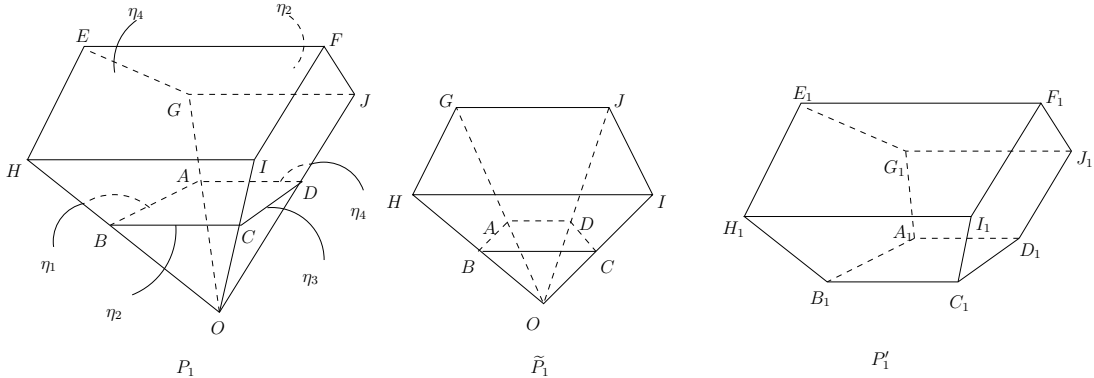

\begin{figure}[ht]
        \centerline{
           \scalebox{0.63}{
            \input{eg4.pstex_t}
            }
          }
       \caption {The edge-simple polytope $P_2$ with the function $\lambda^{(2)}$.}
        \label{eg4}
      \end{figure}
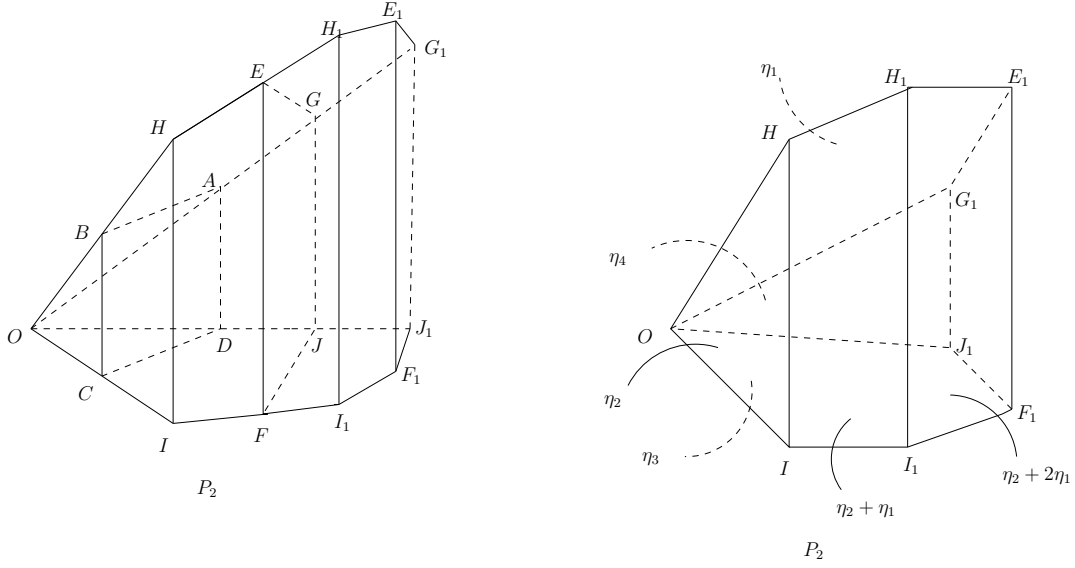

\textbf{For the case } \ref{ca3}: Suppose $ det(\eta_2,\eta_4) = k > 1 $.
Define a function $\lambda^{(1)} $ on the set of facets of $P_1$ except $GEFJ$ by

\begin{equation}
 \begin{array}{ll}  \lambda^{(1)}(OGEH) = \eta_1, ~ \lambda^{(1)}(OHI) = \eta_2, ~ \lambda^{(1)}(OIFJ) = \eta_3,
~ \lambda^{(1)}(OGJ) = \eta_4,\\
\mbox{and} ~ \lambda^{(1)}(EHIF) = \eta_2 + \eta_1.
\end{array}
\end{equation}

So the function $\lambda^{(1)}$ satisfies the condition of an isotropy function
of the edge-simple polytope $P_1$ along each edge except the edges of the rectangle $GEFJ$.
The restriction of the function $\lambda^{(1)}$
on the edges $GE, EF, FJ, GJ$ of the rectangle $GEFJ$ gives the following equations,

\begin{equation}
 \begin{array}{ll} \lvert det[\lambda^{(1)}(GE),\lambda^{(1)}(EF)] \rvert =1, ~
\lvert det[\lambda^{(1)}(EF),\lambda^{(1)}(FJ)]\rvert =1, \\
\lvert det[\lambda^{(1)}(FJ),\lambda^{(1)}(GJ)]\rvert =1, ~ \lvert det[\lambda^{(1)}(GJ),\lambda^{(1)}(GE)]\rvert =1\\
\mbox{and} ~  det[\lambda^{(1)}(EF),\lambda^{(1)}(GJ)] = k-1 < k.
\end{array}
\end{equation}

Let $ P_1^{\prime} $ be a $3$-dimensional convex polytope as in the Figure \ref{eg2}.
Identifying the facet $GEFJ$ of $P_1$ and $A_1B_1C_1D_1$ of $ P_1^{\prime} $ through
a suitable diffeomorphism of manifold with corners such that the vertices $ G, E, F, J$
maps to the vertices $A_1, B_1, C_1, D_1$ respectively, we can form a new convex polytope $P_2$, see Figure \ref{eg4}.
After the identification the following holds.
\begin{enumerate}
\item The facet of $P_1$ containing $GE$ and the facet of $ P_1^{\prime} $ containing $A_1B_1$ make the facet $OHH_1E_1G_1$ of $P_2$.

\item The facet of $P_1$ containing $EF$ and the facet of $ P_1^{\prime} $ containing $B_1C_1$ make the facet $HH_1I_1I$ of $P_2$.

\item The facet of $P_1$ containing $FJ$ and the facet of $ P_1^{\prime} $ containing $C_1D_1$ make the facet $OII_1F_1J_1$ of $P_2$.

\item The facet of $P_1$ containing $JG$ and the facet of $ P_1^{\prime} $ containing $D_1A_1$ make the facet $OJ_1G_1$ of $P_2$.
\end{enumerate}

The polytope $P_2$ is an edge-simple polytope. We define a function $\lambda^{(2)}$ on the set of facets
of $P_2$ except $G_1E_1F_1J_1$ by

\begin{equation}
 \begin{array}{ll} \lambda^{(2)}(OHH_1E_1G_1) = \eta_1, ~ \lambda^{(2)}(OIH) = \eta_2 , ~ \lambda^{(2)}(OII_1F_1J_1) = \eta_3,\\
\lambda^{(2)}(OJ_1G_1) = \eta_4, ~ \lambda^{(2)}(HH_1I_1I) = \eta_2 +\eta_1\\
\mbox{and} ~ \lambda^{(2)}(H_1I_1F_1E_1)= \eta_2 + 2\eta_1.
\end{array}
\end{equation}

So the function $\lambda^{(2)}$ satisfies the condition of an isotropy function of the edge-simple polytope $P_2$ along each edge
except the edges of the rectangle $G_1E_1F_1J_1$. The restriction of the function $\lambda^{(2)}$ on the edges
namely $G_1E_1, E_1F_1, F_1J_1, G_1J_1$ of the rectangle $G_1E_1F_1J_1$ gives the following equations,

\begin{equation}
 \begin{array}{ll}
\lvert det[\lambda^2(G_1E_1),\lambda^2(E_1F_1)]\rvert =1, ~ \lvert det[\lambda^2(E_1F_1),\lambda^2(F_1J_1)]\rvert =1,\\
\lvert det[\lambda^2(F_1J_1),\lambda^2(G_1J_1)]\rvert =1, ~ \lvert det[\lambda^2(G_1J_1),\lambda^2(G_1E_1)]\rvert =1\\
\mbox{and} ~ det[\lambda^2(E_1F_1), \lambda^2(G_1J_1)] = k-2 < k-1.
\end{array}
\end{equation}

Proceeding in this way, at $k$-th step we construct an edge-simple polytope $P_k$ with the function $\lambda^{(k)}$,
extending the function $\lambda^{(k-1)}$, on the set of facets of $P_k$ such that
\begin{equation}
 \begin{array}{ll} \lambda^{(k)}(H_{k-2}H_{k-1}I_{k-1}I_{k-2}) = \eta_2 + (k-1)\eta_1 = \lambda^{(k-1)}(H_{k-2}I_{k-2}F_{k-2}E_{k-2}),\\
\lambda^{(k)}(OG_{k-1}J_{k-1}) = \eta_4 = \lambda^{(k-1)}(OG_{k-2}J_{k-2}),\\
\lambda^{(k)}(H_{k-1}I_{k-1}F_{k-1}E_{k-1})= \eta_4 ~ \mbox{and} ~ \lambda^{(k)}(G_{k-1}E_{k-1}F_{k-1}J_{k-1})= \eta_2 + (k-1) \eta_1.
\end{array}
\end{equation}
Observe that the function $\lambda := \lambda^{(k)}$ is an isotropy function of the edge-simple polytope $P_k$.
So we get an oriented $\TT^2$-manifold with boundary $W(Q_{P_k}, \lambda)$ where the boundary is the quasitoric manifold
$-M \sqcup \sqcup_{k_1} \CP^2 \sqcup \sqcup_{k_2} \bar{\CP^2}$ for some integers $k_1, k_2$. 
Similarly to the previous cases we can construct the following orientation reversing equivariant homeomorphisms.
\begin{equation}\label{hiz3}
 \begin{array}{ll} (\TT^2 \times P_{k_H})/\sim \to (\TT^2 \times P_{k_I})/\sim, ~
(\TT^2 \times P_{1_{G_{k-1}}})/\sim \to (\TT^2 \times P_{1_{J_{k-1}}})/\sim,\\\\
(\TT^2 \times P_{k_{E_{k-1}}})/\sim = (\TT^2 \times P_{k_{F_{k-1}}})/\sim~ \mbox{and}~
(\TT^2 \times P_{k_{H_i}})/\sim \to (\TT^2 \times P_{k_{I_i}})/\sim 
\end{array}
\end{equation}
for $i= 1, \ldots, k-1$. So $k_1 = k_2$. Hence $ [M] = 0 [\CP^2]$.
 Identifying the corresponding boundaries of $ W(Q_{P_k}, \lambda)$
via the equivariant homeomorphisms of equation \ref{hiz3} we get that $M$ is
the boundary of a nice oriented $\TT^2$ manifold.

If $k < -1$, similarly we can show $ [M] = 0 [\CP^2]$ and we can construct nice 
oriented $\TT^2$ manifold with boundary $W$ where the boundary is $M$.

Hence given a Hirzebruch surface $M$ with natural $\TT^2$ action we construct a nice 
$5$-dimensional oriented $\TT^2$ manifold with boundary where the boundary is $M$.
Thus we get the following interesting lemma.
\begin{lemma}\label{hizb}
The $\TT^2$-cobordism class of a Hirzebruch surface is trivial. In particular, oriented
cobordism class of a Hirzebruch surface is also trivial.
\end{lemma}

\begin{figure}[ht]
        \centerline{
           \scalebox{0.67}{
            \input{eg05.pstex_t}
            }
          }
       \caption {The edge-simple polytope $P^{\prime\prime}$ and an isotropy function $\lambda$ associated to the case \ref{ca4}.}
        \label{egca}
      \end{figure}
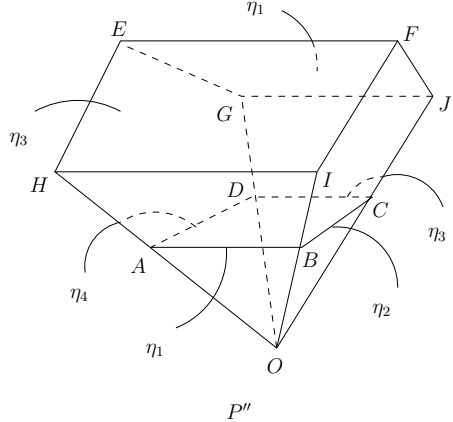

\textbf{For the case } \ref{ca4}: In this case $\lvert det[\eta_1,\eta_3] \rvert =1 $. Following case \ref{ca2},
we can construct an edge simple polytope $P^{\prime\prime}$ and an isotropy function $\lambda$ over this
edge-simple polytope, see Figure \ref{egca}.
Hence we can construct an oriented $\TT^2$ manifold with quasitoric boundary $W(Q_{P^{\prime \prime}}, \lambda)$
where the boundary is $-M \sqcup \sqcup_{k_1} \CP^2 \sqcup \sqcup_{k_2} \bar{\CP^2}$ for some integers $k_1, k_2$.
We may assume that 'the angles between the planes $OHI$ and $HIFE$' and 'the angles between the planes $EFJG$ and $HIFE$'
are equal. Clearly a suitable reflection induces the following orientation reversing equivariant homeomorphisms.
\begin{equation}
 (\TT^2 \times P^{\prime \prime}_H)/\sim \to (\TT^2 \times P^{\prime \prime}_E)/\sim~
\mbox{and} ~(\TT^2 \times P^{\prime \prime}_I)/\sim \to (\TT^2 \times P^{\prime \prime}_F)/\sim.
\end{equation}

Let $\CP^2_J = (\TT^2 \times P_{J}^{\prime \prime})/\sim$ and $\CP^2_G = (\TT^2 \times P_{G}^{\prime \prime})/\sim$.
Observe that the characteristic functions of the triangles $P_{J}^{\prime \prime}$ and $P_{G}^{\prime \prime}$
are differ by a non-trivial automorphism of $\TT^2$ (or $\ZZ^2$). So $\CP^2_J$ and $\CP^2_G$ are complex
projective space $\CP^2$ with two non-equivariant $\TT^2$-actions. Hence $ [M] = [\CP^2_J] + [\CP^2_G]$.

To compute the group $CG_2$ we use the induction on the number of facets of $2$-dimensional simple convex
polytope in $\RR^2$.  We rewrite the proof of well-known following lemma briefly.
\begin{lemma}\label{coblem}
The equivariant connected sum of two quasitoric manifolds is equivariantly cobordant to the disjoint union
of these two quasitoric manifolds.
\end{lemma}
\begin{proof}
 Let $M_1$ and $M_2$ be two quasitoric manifolds of dimension $2n$. 
 Then $W_1 := [0,1] \times M_1$ and $W_2 := [0,1] \times M_2$ are oriented $\TT^n$-manifolds with boundary such that
$$\partial{W_1}= 0 \times (-M_1) \sqcup 1 \times M_1 ~ \mbox{and} ~ \partial{W_2}= 0 \times (-M_2) \sqcup 1 \times M_2.$$
 Let $x_1 \in M_1$ and $x_2 \in M_2$ be two fixed points. Let $ U_1 \subset W_1$ and
$ U_2 \subset W_2$ be two $\TT^n$ invariant open neighborhoods of $1\times x_1$ and $1 \times x_2$ respectively.
Identifying $\partial{U_1} \subset (W_1 - U_1)$ and $\partial{U_2} \subset (W_2 - U_2)$ via a suitable orientation
preserving equivariant homeomorphism we get the lemma.
\end{proof}

Now consider the case of a quasitoric manifold $M$ over a convex $2$-polytope $Q$ with $m$ facets, where $m > 4$.
By the classification result of $4$-dimensional quasitoric manifold which is discussed in Remark \ref{clasq},
$M$ is one of the following equivariant connected sum.
\begin{equation}
 M = N_1 \# \CP^2
\end{equation}
\begin{equation}
 M = N_2 \# \bar{\CP^2}
\end{equation}
\begin{equation}
 M = N_3 \# M_k^4
\end{equation}
The quasitoric manifolds $N_1, N_2$ and $N_3$ are associated to the $2$-polytopes $Q_1, Q_2$ and $Q_3$ respectively.
The number of facets of $Q_1, Q_2$ and $Q_3$ are $m-1,~ m-1$ and $m-2$ respectively.
The quasitoric manifold $M_k^4$ is defined in subsection \ref{defegm}. In previous calculations we
have shown that $[M_k^4] = 0 [\CP^2]$. So by the Lemma \ref{coblem} we get
either $[M] = [N_1] + [\CP^2]$ or $[M] = [N_2] - [\CP^2]$ or $[M] = [N_3] $. Thus using
the induction on $m$, the number of facets of $Q$, we can prove the following.

\begin{lemma}\label{coborlm}
Any $4$-dimensional quasitoric manifold is equivariantly cobordant to some $\TT^2$-cobordism classes of $\CP^2$.
\end{lemma}
We classify the equivariant cobordism classes of all $\TT^2$-actions on $\CP^2$. Let $Q$ be a triangle and
$\{F_1, F_2, F_3\}$ be the edges (facets) of $Q$. Let $\eta : \{F_1, F_2, F_3\} \to \ZZ^2/\ZZ_2$ be a
characteristic function such that $\eta(F_1) = [(a_1, b_1)]$ and $\eta(F_2) = [(a_2, b_2)]$.
We may assume that
$$ det(\eta(F_1), \eta(F_2))= |(a_1, b_1; a_2, b_2)|=1$$
where $(a_1, b_1; a_2, b_2)$ is the $2 \times 2$ matrix in $SL(2, \ZZ)$ with row vectors $\eta(F_1) $ and $\eta(F_2)$.
We denote this matrix by $\eta$ also.
Then either $\eta(F_3) = [(a_1+a_2, b_1+b_2)]$ or $\eta(F_3) = [(a_1-a_2, b_1-b_2)]$. Let $\eta^{\prime}$
and $\eta^{\prime \prime}$ be two characteristic function defined respectively by,
$$\eta^{\prime}(F_1) = [(a_1, b_1)], \eta^{\prime}(F_2) = [(a_2, b_2)], \eta^{\prime}(F_3) = [(a_1+a_2, b_1 + b_2)]$$
and
$$\eta^{\prime \prime}(F_1) = [(a_1, b_1)], \eta^{\prime \prime}(F_2) = [(a_2, b_2)], \eta^{\prime \prime}(F_3) = [(a_1-a_2, b_1 - b_2)].$$
Denote the quasitoric manifolds associated to the pairs $(Q, \eta^{\prime})$ and $(Q, \eta^{\prime \prime})$ by 
$\CP^2_{\eta^{\prime}}$ and $\CP^2_{\eta^{\prime \prime}}$ respectively. Define an equivalence relation $\sim_{eq}$ on $SL(2, \ZZ)$ by 
$$(a_1, b_1; a_2, b_2) \sim_{eq} (-a_1, -b_1; -a_2, -b_2).$$
Denote the equivalence class of $\eta \in SL(2, \ZZ)$ by $[\eta]_{eq}$. Observe that if $[\eta_1]_{eq} \neq [\eta_2]_{eq}$
then the corresponding characteristic functions are differ by $\delta_{\ast}$, for some non trivial auto morphism
$\delta: \TT^2 \to \TT^2$. Using Lemma \ref{clema2} we get the following classification.
\begin{lemma}
 A $\TT^2$-actions on $\CP^2$ is equivariantly homeomorphic to either $\CP^2_{\eta^{\prime}}$ or $\CP^2_{\eta^{\prime \prime}}$
for a unique $[\eta]_{eq} \in SL(2, \ZZ)/ \sim_{eq}$.
\end{lemma}
Note that the natural $\TT^2$-actions on $\CP^2_{\eta^{\prime}}$ and $\CP^2_{\eta^{\prime \prime}}$ are same.
Consider the linear map $L_{\eta} : \ZZ^2 \to \ZZ^2$, defined by $L_{\eta}(1,0)= (a_1, b_1), L_{\eta}(0, 1) = (a_2, b_2)$.
The map $L_{\eta}$ induces orientation preserving homeomorphisms $\CP^2_s \to \CP^2_{\eta^{\prime}}$ and
$\bar{\CP^2}_s \to \CP^2_{\eta^{\prime \prime}}$. Thus,
\begin{lemma}\label{orit2}
The oriented $\TT^2$-cobordism class of a $\TT^2$-action on $\CP^2$ is $[\CP^2_{\eta^{\prime}}]$
for a unique $[\eta]_{eq} \in SL(2, \ZZ)/ \sim_{eq}$.
\end{lemma}

Since the order of oriented cobordism class of $\CP^2$ is infinite, we get the following theorem.
\begin{theorem}\label{coborthm}
The oriented torus cobordism group $CG_2$ is an infinite abelian group with a set of generators
$\{ [\CP^2_{\eta^{\prime}}] ~:~ [\eta]_{eq} \in SL(2, \ZZ)/ \sim_{eq} \}$.
\end{theorem}
We do not know whether these are the free generators:
\begin{qus}
 Describe all the relations among the generators given in Theorem \ref{coborthm}.
\end{qus}

We discuss the actions of $\TT^n$ on $(2n+1)$-dimensional manifolds (possibly with boundary) where the actions are
similar to the locally standard actions. We again call these actions $locally ~ standard ~ actions$. We discuss
some properties of these actions explicitly in our next article 'Odd dimensional torus manifolds'. Let $\rho_s$
be the standard action of $\TT^n$ on $\CC^n$. Consider the action $\rho$ of $\TT^n$ on $\CC^n \times \RR$ defined
by $\rho(t, (z, r))= (\rho_s(t, z) ,r)$.
\begin{defn}\label{lc}
A smooth action of $\TT^n$ on a $(2n+1)$-dimensional smooth manifold (possibly with boundary) $W$ is said to be locally
standard if every point $a \in W $ has a $\TT^n$-stable open neighborhood
$W_a$ and a diffeomorphism $\xi_a : W_a \to V_a$, where $V_a$ is a $\TT^n$-stable
open subset of $\CC^n \times \RR_{\geq 0}$ under the action $\rho$, and an isomorphism $\delta_a : \TT^n \to \TT^n$ such that
\ $\xi_a (t\cdot x) = \rho(\delta_a (t), \xi_a(x))$ for all $(t,x) \in \TT^n \times W_a$.
\end{defn}
\begin{example}
 Consider $S^{2n+1} =\{(z_1, \ldots, z_{n+1}) \in \CC^{n+1} : \Sigma_i |z_i|^2 = 1\}$, the torus $\TT^n$ acts on $S^{2n+1}$ by
$(t_1, \ldots, t_n) \cdot (z_1, \ldots, z_n, z_{n+1}) =(t_1z_1, \ldots, t_nz_n, z_{n+1})$. This action is a locally standard
action in the sense of Definition \ref{lc}. When $n=1$, the orbit space of this action is a closed $2$-dimensional disk.
\end{example}
\begin{remark}
 If $W$ is a $(2n+1)$-dimensional smooth manifold with boundary with a locally standard action of $\TT^n$, the fixed point set
$W^{\TT^{n}}$ is a disjoint union of some circles and closed intervals.
\end{remark}


{\bf Acknowledgement.} I would like to thank Prof. Mikiya Masuda and my supervisor Prof. Mainak Poddar for helpful suggestions
and stimulating discussions. The author is thankful to the anonymous referee
for helpful suggestions. I thank I.S.I. Kolkata for supporting my research fellowship. This work was partially supported
by the National Research Foundation of Korea (NRF) grant funded by Korea government (MEST) (No. 2012-0000795).


\renewcommand{\refname}{References}

\vspace{1cm}

\vfill

\end{document}

%% file: egch001.pstex_t
\begin{picture}(0,0)%
\includegraphics{egch001.pstex}%
\end{picture}%
\setlength{\unitlength}{4144sp}%
\begingroup\makeatletter\ifx\SetFigFont\undefined%
\gdef\SetFigFont#1#2#3#4#5{%
  \reset@font\fontsize{#1}{#2pt}%
  \fontfamily{#3}\fontseries{#4}\fontshape{#5}%
  \selectfont}%
\fi\endgroup%
\begin{picture}(5655,1938)(1111,-3595)
\put(2116,-3481){\makebox(0,0)[lb]{\smash{{\SetFigFont{12}{14.4}{\rmdefault}{\mddefault}{\updefault}{\color[rgb]{0,0,0}$(1,0)$}%
}}}}
\put(2926,-2446){\makebox(0,0)[lb]{\smash{{\SetFigFont{12}{14.4}{\rmdefault}{\mddefault}{\updefault}{\color[rgb]{0,0,0}$$(1,1)}%
}}}}
\put(5581,-3526){\makebox(0,0)[lb]{\smash{{\SetFigFont{12}{14.4}{\rmdefault}{\mddefault}{\updefault}{\color[rgb]{0,0,0}$(1,0)$}%
}}}}
\put(4906,-2491){\makebox(0,0)[lb]{\smash{{\SetFigFont{12}{14.4}{\rmdefault}{\mddefault}{\updefault}{\color[rgb]{0,0,0}$(0,1)$}%
}}}}
\put(6436,-2626){\makebox(0,0)[lb]{\smash{{\SetFigFont{12}{14.4}{\rmdefault}{\mddefault}{\updefault}{\color[rgb]{0,0,0}$(1,-1)$}%
}}}}
\put(1261,-2536){\makebox(0,0)[lb]{\smash{{\SetFigFont{12}{14.4}{\rmdefault}{\mddefault}{\updefault}{\color[rgb]{0,0,0}$(0,1)$}%
}}}}
\put(2161,-1816){\makebox(0,0)[lb]{\smash{{\SetFigFont{12}{14.4}{\rmdefault}{\mddefault}{\updefault}{\color[rgb]{0,0,0}$A_1$}%
}}}}
\put(1126,-3481){\makebox(0,0)[lb]{\smash{{\SetFigFont{12}{14.4}{\rmdefault}{\mddefault}{\updefault}{\color[rgb]{0,0,0}$B_1$}%
}}}}
\put(3601,-3481){\makebox(0,0)[lb]{\smash{{\SetFigFont{12}{14.4}{\rmdefault}{\mddefault}{\updefault}{\color[rgb]{0,0,0}$C_1$}%
}}}}
\put(5761,-1816){\makebox(0,0)[lb]{\smash{{\SetFigFont{12}{14.4}{\rmdefault}{\mddefault}{\updefault}{\color[rgb]{0,0,0}$A_2$}%
}}}}
\put(4861,-3436){\makebox(0,0)[lb]{\smash{{\SetFigFont{12}{14.4}{\rmdefault}{\mddefault}{\updefault}{\color[rgb]{0,0,0}$B_2$}%
}}}}
\put(6751,-3436){\makebox(0,0)[lb]{\smash{{\SetFigFont{12}{14.4}{\rmdefault}{\mddefault}{\updefault}{\color[rgb]{0,0,0}$C_2$}%
}}}}
\end{picture}%

%% file: egch002.pstex_t
\begin{picture}(0,0)%
\includegraphics{egch002.pstex}%
\end{picture}%
\setlength{\unitlength}{4144sp}%
\begingroup\makeatletter\ifx\SetFigFont\undefined%
\gdef\SetFigFont#1#2#3#4#5{%
  \reset@font\fontsize{#1}{#2pt}%
  \fontfamily{#3}\fontseries{#4}\fontshape{#5}%
  \selectfont}%
\fi\endgroup%
\begin{picture}(7050,2478)(706,-4045)
\put(1801,-3931){\makebox(0,0)[lb]{\smash{{\SetFigFont{12}{14.4}{\rmdefault}{\mddefault}{\updefault}{\color[rgb]{0,0,0}$(1,0)$}%
}}}}
\put(721,-2806){\makebox(0,0)[lb]{\smash{{\SetFigFont{12}{14.4}{\rmdefault}{\mddefault}{\updefault}{\color[rgb]{0,0,0}$(0,1)$}%
}}}}
\put(3241,-2941){\makebox(0,0)[lb]{\smash{{\SetFigFont{12}{14.4}{\rmdefault}{\mddefault}{\updefault}{\color[rgb]{0,0,0}$(0,1)$}%
}}}}
\put(6436,-3976){\makebox(0,0)[lb]{\smash{{\SetFigFont{12}{14.4}{\rmdefault}{\mddefault}{\updefault}{\color[rgb]{0,0,0}$(1,0)$}%
}}}}
\put(5176,-2896){\makebox(0,0)[lb]{\smash{{\SetFigFont{12}{14.4}{\rmdefault}{\mddefault}{\updefault}{\color[rgb]{0,0,0}$(0,1)$}%
}}}}
\put(7741,-2716){\makebox(0,0)[lb]{\smash{{\SetFigFont{12}{14.4}{\rmdefault}{\mddefault}{\updefault}{\color[rgb]{0,0,0}$(-1,1)$}%
}}}}
\put(6481,-1726){\makebox(0,0)[lb]{\smash{{\SetFigFont{12}{14.4}{\rmdefault}{\mddefault}{\updefault}{\color[rgb]{0,0,0}$(1,-2)$}%
}}}}
\put(1891,-1726){\makebox(0,0)[lb]{\smash{{\SetFigFont{12}{14.4}{\rmdefault}{\mddefault}{\updefault}{\color[rgb]{0,0,0}$(1,k)$}%
}}}}
\put(1036,-1951){\makebox(0,0)[lb]{\smash{{\SetFigFont{12}{14.4}{\rmdefault}{\mddefault}{\updefault}{\color[rgb]{0,0,0}$A$}%
}}}}
\put(1081,-3886){\makebox(0,0)[lb]{\smash{{\SetFigFont{12}{14.4}{\rmdefault}{\mddefault}{\updefault}{\color[rgb]{0,0,0}$B$}%
}}}}
\put(3106,-3886){\makebox(0,0)[lb]{\smash{{\SetFigFont{12}{14.4}{\rmdefault}{\mddefault}{\updefault}{\color[rgb]{0,0,0}$C$}%
}}}}
\put(3196,-1951){\makebox(0,0)[lb]{\smash{{\SetFigFont{12}{14.4}{\rmdefault}{\mddefault}{\updefault}{\color[rgb]{0,0,0}$D$}%
}}}}
\put(5581,-1996){\makebox(0,0)[lb]{\smash{{\SetFigFont{12}{14.4}{\rmdefault}{\mddefault}{\updefault}{\color[rgb]{0,0,0}$A$}%
}}}}
\put(5716,-3976){\makebox(0,0)[lb]{\smash{{\SetFigFont{12}{14.4}{\rmdefault}{\mddefault}{\updefault}{\color[rgb]{0,0,0}$B$}%
}}}}
\put(7696,-3976){\makebox(0,0)[lb]{\smash{{\SetFigFont{12}{14.4}{\rmdefault}{\mddefault}{\updefault}{\color[rgb]{0,0,0}$C$}%
}}}}
\put(7741,-1906){\makebox(0,0)[lb]{\smash{{\SetFigFont{12}{14.4}{\rmdefault}{\mddefault}{\updefault}{\color[rgb]{0,0,0}$D$}%
}}}}
\end{picture}%

%% file: eg1.pstex_t
\begin{picture}(0,0)%
\includegraphics{eg1.pstex}%
\end{picture}%
\setlength{\unitlength}{4144sp}%
\begingroup\makeatletter\ifx\SetFigFontNFSS\undefined%
\gdef\SetFigFontNFSS#1#2#3#4#5{%
  \reset@font\fontsize{#1}{#2pt}%
  \fontfamily{#3}\fontseries{#4}\fontshape{#5}%
  \selectfont}%
\fi\endgroup%
\begin{picture}(7224,1689)(889,-4348)
\end{picture}%

%% file: egc1.pstex_t
\begin{picture}(0,0)%
\includegraphics{egc1.pstex}%
\end{picture}%
\setlength{\unitlength}{3947sp}%
\begingroup\makeatletter\ifx\SetFigFontNFSS\undefined%
\gdef\SetFigFontNFSS#1#2#3#4#5{%
  \reset@font\fontsize{#1}{#2pt}%
  \fontfamily{#3}\fontseries{#4}\fontshape{#5}%
  \selectfont}%
\fi\endgroup%
\begin{picture}(9780,5253)(2836,-4555)
\put(2851,-2311){\makebox(0,0)[lb]{\smash{{\SetFigFontNFSS{12}{14.4}{\rmdefault}{\mddefault}{\updefault}{\color[rgb]{0,0,0}$(1,1)$}%
}}}}
\put(3151,-3586){\makebox(0,0)[lb]{\smash{{\SetFigFontNFSS{12}{14.4}{\rmdefault}{\mddefault}{\updefault}{\color[rgb]{0,0,0}$(1,0)$}%
}}}}
\put(4426,-3511){\makebox(0,0)[lb]{\smash{{\SetFigFontNFSS{12}{14.4}{\rmdefault}{\mddefault}{\updefault}{\color[rgb]{0,0,0}$(0,1)$}%
}}}}
\put(5326,-3136){\makebox(0,0)[lb]{\smash{{\SetFigFontNFSS{12}{14.4}{\rmdefault}{\mddefault}{\updefault}{\color[rgb]{0,0,0}$(1,1)$}%
}}}}
\put(5701,-1486){\makebox(0,0)[lb]{\smash{{\SetFigFontNFSS{12}{14.4}{\rmdefault}{\mddefault}{\updefault}{\color[rgb]{0,0,0}$(1,0)$}%
}}}}
\put(2851,-1036){\makebox(0,0)[lb]{\smash{{\SetFigFontNFSS{12}{14.4}{\rmdefault}{\mddefault}{\updefault}{\color[rgb]{0,0,0}$(0,1)$}%
}}}}
\put(3826,-3961){\makebox(0,0)[lb]{\smash{{\SetFigFontNFSS{12}{14.4}{\rmdefault}{\mddefault}{\updefault}{\color[rgb]{0,0,0}$(I^3, \lambda)$}%
}}}}
\put(8551,-4111){\makebox(0,0)[lb]{\smash{{\SetFigFontNFSS{12}{14.4}{\rmdefault}{\mddefault}{\updefault}{\color[rgb]{0,0,0}$(1,0)$}%
}}}}
\put(9826,-3961){\makebox(0,0)[lb]{\smash{{\SetFigFontNFSS{12}{14.4}{\rmdefault}{\mddefault}{\updefault}{\color[rgb]{0,0,0}$(0,1)$}%
}}}}
\put(11701,-3286){\makebox(0,0)[lb]{\smash{{\SetFigFontNFSS{12}{14.4}{\rmdefault}{\mddefault}{\updefault}{\color[rgb]{0,0,0}$(1,1)$}%
}}}}
\put(3301,-2986){\makebox(0,0)[lb]{\smash{{\SetFigFontNFSS{12}{14.4}{\rmdefault}{\mddefault}{\updefault}{\color[rgb]{0,0,0}$v_1$}%
}}}}
\put(4801,-3061){\makebox(0,0)[lb]{\smash{{\SetFigFontNFSS{12}{14.4}{\rmdefault}{\mddefault}{\updefault}{\color[rgb]{0,0,0}$v_2$}%
}}}}
\put(4951,-1636){\makebox(0,0)[lb]{\smash{{\SetFigFontNFSS{12}{14.4}{\rmdefault}{\mddefault}{\updefault}{\color[rgb]{0,0,0}$v_3$}%
}}}}
\put(3301,-1636){\makebox(0,0)[lb]{\smash{{\SetFigFontNFSS{12}{14.4}{\rmdefault}{\mddefault}{\updefault}{\color[rgb]{0,0,0}$v_4$}%
}}}}
\put(4276,-2086){\makebox(0,0)[lb]{\smash{{\SetFigFontNFSS{12}{14.4}{\rmdefault}{\mddefault}{\updefault}{\color[rgb]{0,0,0}$v_5$}%
}}}}
\put(5476,-2236){\makebox(0,0)[lb]{\smash{{\SetFigFontNFSS{12}{14.4}{\rmdefault}{\mddefault}{\updefault}{\color[rgb]{0,0,0}$v_6$}%
}}}}
\put(5476,-1036){\makebox(0,0)[lb]{\smash{{\SetFigFontNFSS{12}{14.4}{\rmdefault}{\mddefault}{\updefault}{\color[rgb]{0,0,0}$v_7$}%
}}}}
\put(4126,-811){\makebox(0,0)[lb]{\smash{{\SetFigFontNFSS{12}{14.4}{\rmdefault}{\mddefault}{\updefault}{\color[rgb]{0,0,0}$v_8$}%
}}}}
\put(7651,-1936){\makebox(0,0)[lb]{\smash{{\SetFigFontNFSS{12}{14.4}{\rmdefault}{\mddefault}{\updefault}{\color[rgb]{0,0,0}$(1,1)$}%
}}}}
\put(10201,539){\makebox(0,0)[lb]{\smash{{\SetFigFontNFSS{12}{14.4}{\rmdefault}{\mddefault}{\updefault}{\color[rgb]{0,0,0}$(0,1)$}%
}}}}
\put(11401,314){\makebox(0,0)[lb]{\smash{{\SetFigFontNFSS{12}{14.4}{\rmdefault}{\mddefault}{\updefault}{\color[rgb]{0,0,0}$(1,0)$}%
}}}}
\put(9451,-4486){\makebox(0,0)[lb]{\smash{{\SetFigFontNFSS{12}{14.4}{\rmdefault}{\mddefault}{\updefault}{\color[rgb]{0,0,0}$(Q_{I^3}, \lambda)$}%
}}}}
\put(7726,-3661){\makebox(0,0)[lb]{\smash{{\SetFigFontNFSS{12}{14.4}{\rmdefault}{\mddefault}{\updefault}{\color[rgb]{0,0,0}$P_{H_1}$}%
}}}}
\put(11176,-3961){\makebox(0,0)[lb]{\smash{{\SetFigFontNFSS{12}{14.4}{\rmdefault}{\mddefault}{\updefault}{\color[rgb]{0,0,0}$P_{H_2}$}%
}}}}
\put(12601,-1261){\makebox(0,0)[lb]{\smash{{\SetFigFontNFSS{12}{14.4}{\rmdefault}{\mddefault}{\updefault}{\color[rgb]{0,0,0}$P_{H_3}$}%
}}}}
\put(7726,-736){\makebox(0,0)[lb]{\smash{{\SetFigFontNFSS{12}{14.4}{\rmdefault}{\mddefault}{\updefault}{\color[rgb]{0,0,0}$P_{H_4}$}%
}}}}
\put(7501,-2761){\makebox(0,0)[lb]{\smash{{\SetFigFontNFSS{12}{14.4}{\rmdefault}{\mddefault}{\updefault}{\color[rgb]{0,0,0}$P_{H_5}$}%
}}}}
\put(12376,-2611){\makebox(0,0)[lb]{\smash{{\SetFigFontNFSS{12}{14.4}{\rmdefault}{\mddefault}{\updefault}{\color[rgb]{0,0,0}$P_{H_6}$}%
}}}}
\put(12451,-61){\makebox(0,0)[lb]{\smash{{\SetFigFontNFSS{12}{14.4}{\rmdefault}{\mddefault}{\updefault}{\color[rgb]{0,0,0}$P_{H_7}$}%
}}}}
\put(9226,389){\makebox(0,0)[lb]{\smash{{\SetFigFontNFSS{12}{14.4}{\rmdefault}{\mddefault}{\updefault}{\color[rgb]{0,0,0}$P_{H_8}$}%
}}}}
\end{picture}%

%% file: egc2.pstex_t
\begin{picture}(0,0)%
\includegraphics{egc2.pstex}%
\end{picture}%
\setlength{\unitlength}{3947sp}%
\begingroup\makeatletter\ifx\SetFigFontNFSS\undefined%
\gdef\SetFigFontNFSS#1#2#3#4#5{%
  \reset@font\fontsize{#1}{#2pt}%
  \fontfamily{#3}\fontseries{#4}\fontshape{#5}%
  \selectfont}%
\fi\endgroup%
\begin{picture}(8730,3978)(361,-5305)
\put(976,-4261){\makebox(0,0)[lb]{\smash{{\SetFigFontNFSS{12}{14.4}{\rmdefault}{\mddefault}{\updefault}{\color[rgb]{0,0,0}$v_1$}%
}}}}
\put(2926,-4261){\makebox(0,0)[lb]{\smash{{\SetFigFontNFSS{12}{14.4}{\rmdefault}{\mddefault}{\updefault}{\color[rgb]{0,0,0}$v_2$}%
}}}}
\put(3376,-2836){\makebox(0,0)[lb]{\smash{{\SetFigFontNFSS{12}{14.4}{\rmdefault}{\mddefault}{\updefault}{\color[rgb]{0,0,0}$v_3$}%
}}}}
\put(601,-2761){\makebox(0,0)[lb]{\smash{{\SetFigFontNFSS{12}{14.4}{\rmdefault}{\mddefault}{\updefault}{\color[rgb]{0,0,0}$v_4$}%
}}}}
\put(1726,-1486){\makebox(0,0)[lb]{\smash{{\SetFigFontNFSS{12}{14.4}{\rmdefault}{\mddefault}{\updefault}{\color[rgb]{0,0,0}$v_5$}%
}}}}
\put(7126,-4711){\makebox(0,0)[lb]{\smash{{\SetFigFontNFSS{12}{14.4}{\rmdefault}{\mddefault}{\updefault}{\color[rgb]{0,0,0}$(1,1)$}%
}}}}
\put(6076,-4486){\makebox(0,0)[lb]{\smash{{\SetFigFontNFSS{12}{14.4}{\rmdefault}{\mddefault}{\updefault}{\color[rgb]{0,0,0}$(1,0)$}%
}}}}
\put(8401,-2161){\makebox(0,0)[lb]{\smash{{\SetFigFontNFSS{12}{14.4}{\rmdefault}{\mddefault}{\updefault}{\color[rgb]{0,0,0}$(0,1)$}%
}}}}
\put(5101,-1786){\makebox(0,0)[lb]{\smash{{\SetFigFontNFSS{12}{14.4}{\rmdefault}{\mddefault}{\updefault}{\color[rgb]{0,0,0}$(0,1)$}%
}}}}
\put(6376,-1486){\makebox(0,0)[lb]{\smash{{\SetFigFontNFSS{12}{14.4}{\rmdefault}{\mddefault}{\updefault}{\color[rgb]{0,0,0}$(1,0)$}%
}}}}
\put(7576,-1486){\makebox(0,0)[lb]{\smash{{\SetFigFontNFSS{12}{14.4}{\rmdefault}{\mddefault}{\updefault}{\color[rgb]{0,0,0}$P_{H_5}$}%
}}}}
\put(4951,-2386){\makebox(0,0)[lb]{\smash{{\SetFigFontNFSS{12}{14.4}{\rmdefault}{\mddefault}{\updefault}{\color[rgb]{0,0,0}$P_{H_4}$}%
}}}}
\put(5101,-4186){\makebox(0,0)[lb]{\smash{{\SetFigFontNFSS{12}{14.4}{\rmdefault}{\mddefault}{\updefault}{\color[rgb]{0,0,0}$P_{H_1}$}%
}}}}
\put(8176,-4411){\makebox(0,0)[lb]{\smash{{\SetFigFontNFSS{12}{14.4}{\rmdefault}{\mddefault}{\updefault}{\color[rgb]{0,0,0}$P_{H_2}$}%
}}}}
\put(9076,-3286){\makebox(0,0)[lb]{\smash{{\SetFigFontNFSS{12}{14.4}{\rmdefault}{\mddefault}{\updefault}{\color[rgb]{0,0,0}$P_{H_3}$}%
}}}}
\put(1801,-4861){\makebox(0,0)[lb]{\smash{{\SetFigFontNFSS{12}{14.4}{\rmdefault}{\mddefault}{\updefault}{\color[rgb]{0,0,0}$(P_0, \lambda)$}%
}}}}
\put(6451,-5236){\makebox(0,0)[lb]{\smash{{\SetFigFontNFSS{12}{14.4}{\rmdefault}{\mddefault}{\updefault}{\color[rgb]{0,0,0}$(Q_{P_0}, \lambda)$}%
}}}}
\put(376,-3586){\makebox(0,0)[lb]{\smash{{\SetFigFontNFSS{12}{14.4}{\rmdefault}{\mddefault}{\updefault}{\color[rgb]{0,0,0}$(0,1)$}%
}}}}
\put(1651,-4486){\makebox(0,0)[lb]{\smash{{\SetFigFontNFSS{12}{14.4}{\rmdefault}{\mddefault}{\updefault}{\color[rgb]{0,0,0}$(1,0)$}%
}}}}
\put(2926,-4636){\makebox(0,0)[lb]{\smash{{\SetFigFontNFSS{12}{14.4}{\rmdefault}{\mddefault}{\updefault}{\color[rgb]{0,0,0}$(1,1)$}%
}}}}
\put(3451,-3586){\makebox(0,0)[lb]{\smash{{\SetFigFontNFSS{12}{14.4}{\rmdefault}{\mddefault}{\updefault}{\color[rgb]{0,0,0}$(0,1)$}%
}}}}
\put(3001,-1936){\makebox(0,0)[lb]{\smash{{\SetFigFontNFSS{12}{14.4}{\rmdefault}{\mddefault}{\updefault}{\color[rgb]{0,0,0}$(1,0)$}%
}}}}
\end{picture}%

%% file: egoricob.pstex_t
\begin{picture}(0,0)%
\includegraphics{egoricob.pstex}%
\end{picture}%
\setlength{\unitlength}{3947sp}%
\begingroup\makeatletter\ifx\SetFigFontNFSS\undefined%
\gdef\SetFigFontNFSS#1#2#3#4#5{%
  \reset@font\fontsize{#1}{#2pt}%
  \fontfamily{#3}\fontseries{#4}\fontshape{#5}%
  \selectfont}%
\fi\endgroup%
\begin{picture}(4223,3453)(4936,-4780)
\put(7576,-1486){\makebox(0,0)[lb]{\smash{{\SetFigFontNFSS{12}{14.4}{\rmdefault}{\mddefault}{\updefault}{\color[rgb]{0,0,0}$P_{H_5}$}%
}}}}
\put(8176,-4411){\makebox(0,0)[lb]{\smash{{\SetFigFontNFSS{12}{14.4}{\rmdefault}{\mddefault}{\updefault}{\color[rgb]{0,0,0}$P_{H_2}$}%
}}}}
\put(5251,-3736){\makebox(0,0)[lb]{\smash{{\SetFigFontNFSS{12}{14.4}{\rmdefault}{\mddefault}{\updefault}{\color[rgb]{0,0,0}$v_2$}%
}}}}
\put(7651,-4186){\makebox(0,0)[lb]{\smash{{\SetFigFontNFSS{12}{14.4}{\rmdefault}{\mddefault}{\updefault}{\color[rgb]{0,0,0}$v_3$}%
}}}}
\put(8626,-3886){\makebox(0,0)[lb]{\smash{{\SetFigFontNFSS{12}{14.4}{\rmdefault}{\mddefault}{\updefault}{\color[rgb]{0,0,0}$v_4$}%
}}}}
\put(8176,-3511){\makebox(0,0)[lb]{\smash{{\SetFigFontNFSS{12}{14.4}{\rmdefault}{\mddefault}{\updefault}{\color[rgb]{0,0,0}$v_5$}%
}}}}
\put(8476,-2761){\makebox(0,0)[lb]{\smash{{\SetFigFontNFSS{12}{14.4}{\rmdefault}{\mddefault}{\updefault}{\color[rgb]{0,0,0}$v_7$}%
}}}}
\put(5926,-3511){\makebox(0,0)[lb]{\smash{{\SetFigFontNFSS{12}{14.4}{\rmdefault}{\mddefault}{\updefault}{\color[rgb]{0,0,0}$v_1$}%
}}}}
\put(5101,-4186){\makebox(0,0)[lb]{\smash{{\SetFigFontNFSS{12}{14.4}{\rmdefault}{\mddefault}{\updefault}{\color[rgb]{0,0,0}$P_{H_1}$}%
}}}}
\put(6076,-4186){\makebox(0,0)[lb]{\smash{{\SetFigFontNFSS{12}{14.4}{\rmdefault}{\mddefault}{\updefault}{\color[rgb]{0,0,0}$v_0$}%
}}}}
\put(7951,-3136){\makebox(0,0)[lb]{\smash{{\SetFigFontNFSS{12}{14.4}{\rmdefault}{\mddefault}{\updefault}{\color[rgb]{0,0,0}$v_8$}%
}}}}
\put(5026,-3211){\makebox(0,0)[lb]{\smash{{\SetFigFontNFSS{12}{14.4}{\rmdefault}{\mddefault}{\updefault}{\color[rgb]{0,0,0}$v_9$}%
}}}}
\put(5701,-2911){\makebox(0,0)[lb]{\smash{{\SetFigFontNFSS{12}{14.4}{\rmdefault}{\mddefault}{\updefault}{\color[rgb]{0,0,0}$v_{10}$}%
}}}}
\put(5101,-2761){\makebox(0,0)[lb]{\smash{{\SetFigFontNFSS{12}{14.4}{\rmdefault}{\mddefault}{\updefault}{\color[rgb]{0,0,0}$v_{11}$}%
}}}}
\put(7576,-1786){\makebox(0,0)[lb]{\smash{{\SetFigFontNFSS{12}{14.4}{\rmdefault}{\mddefault}{\updefault}{\color[rgb]{0,0,0}$v_{14}$}%
}}}}
\put(6076,-1636){\makebox(0,0)[lb]{\smash{{\SetFigFontNFSS{12}{14.4}{\rmdefault}{\mddefault}{\updefault}{\color[rgb]{0,0,0}$v_{15}$}%
}}}}
\put(7276,-2161){\makebox(0,0)[lb]{\smash{{\SetFigFontNFSS{12}{14.4}{\rmdefault}{\mddefault}{\updefault}{\color[rgb]{0,0,0}$v_{13}$}%
}}}}
\put(6601,-2311){\makebox(0,0)[lb]{\smash{{\SetFigFontNFSS{12}{14.4}{\rmdefault}{\mddefault}{\updefault}{\color[rgb]{0,0,0}$v_{12}$}%
}}}}
\put(8101,-2311){\makebox(0,0)[lb]{\smash{{\SetFigFontNFSS{12}{14.4}{\rmdefault}{\mddefault}{\updefault}{\color[rgb]{0,0,0}$e_{v_{14}}$}%
}}}}
\put(5101,-3436){\makebox(0,0)[lb]{\smash{{\SetFigFontNFSS{12}{14.4}{\rmdefault}{\mddefault}{\updefault}{\color[rgb]{0,0,0}$e_{v_9}$}%
}}}}
\put(4951,-2386){\makebox(0,0)[lb]{\smash{{\SetFigFontNFSS{12}{14.4}{\rmdefault}{\mddefault}{\updefault}{\color[rgb]{0,0,0}$P_{H_4}$}%
}}}}
\put(5551,-2161){\makebox(0,0)[lb]{\smash{{\SetFigFontNFSS{12}{14.4}{\rmdefault}{\mddefault}{\updefault}{\color[rgb]{0,0,0}$e_{v_{15}}$}%
}}}}
\put(6376,-2686){\makebox(0,0)[lb]{\smash{{\SetFigFontNFSS{12}{14.4}{\rmdefault}{\mddefault}{\updefault}{\color[rgb]{0,0,0}$e_{v_{12}}$}%
}}}}
\put(7651,-2761){\makebox(0,0)[lb]{\smash{{\SetFigFontNFSS{12}{14.4}{\rmdefault}{\mddefault}{\updefault}{\color[rgb]{0,0,0}$e_{v_{13}}$}%
}}}}
\put(8701,-3286){\makebox(0,0)[lb]{\smash{{\SetFigFontNFSS{12}{14.4}{\rmdefault}{\mddefault}{\updefault}{\color[rgb]{0,0,0}$v_6$}%
}}}}
\put(9076,-3286){\makebox(0,0)[lb]{\smash{{\SetFigFontNFSS{12}{14.4}{\rmdefault}{\mddefault}{\updefault}{\color[rgb]{0,0,0}$P_{H_3}$}%
}}}}
\put(8701,-3586){\makebox(0,0)[lb]{\smash{{\SetFigFontNFSS{12}{14.4}{\rmdefault}{\mddefault}{\updefault}{\color[rgb]{0,0,0}$e_{v_6}$}%
}}}}
\put(6751,-3886){\makebox(0,0)[lb]{\smash{{\SetFigFontNFSS{12}{14.4}{\rmdefault}{\mddefault}{\updefault}{\color[rgb]{0,0,0}$e_{v_3}$}%
}}}}
\put(6676,-2836){\makebox(0,0)[lb]{\smash{{\SetFigFontNFSS{12}{14.4}{\rmdefault}{\mddefault}{\updefault}{\color[rgb]{0,0,0}$e_{v_8}$}%
}}}}
\put(6376,-4711){\makebox(0,0)[lb]{\smash{{\SetFigFontNFSS{12}{14.4}{\rmdefault}{\mddefault}{\updefault}{\color[rgb]{0,0,0}$Q_{P_{0}}$}%
}}}}
\end{picture}%

%% file: eg2.pstex_t
\begin{picture}(0,0)%
\includegraphics{eg2.pstex}%
\end{picture}%
\setlength{\unitlength}{4144sp}%
\begingroup\makeatletter\ifx\SetFigFontNFSS\undefined%
\gdef\SetFigFontNFSS#1#2#3#4#5{%
  \reset@font\fontsize{#1}{#2pt}%
  \fontfamily{#3}\fontseries{#4}\fontshape{#5}%
  \selectfont}%
\fi\endgroup%
\begin{picture}(3713,3963)(7006,-5260)
\put(9136,-4786){\makebox(0,0)[lb]{\smash{{\SetFigFontNFSS{12}{14.4}{\rmdefault}{\mddefault}{\updefault}{\color[rgb]{0,0,0}$O$}%
}}}}
\put(7921,-3886){\makebox(0,0)[lb]{\smash{{\SetFigFontNFSS{12}{14.4}{\rmdefault}{\mddefault}{\updefault}{\color[rgb]{0,0,0}$B$}%
}}}}
\put(9451,-3841){\makebox(0,0)[lb]{\smash{{\SetFigFontNFSS{12}{14.4}{\rmdefault}{\mddefault}{\updefault}{\color[rgb]{0,0,0}$C$}%
}}}}
\put(10081,-3391){\makebox(0,0)[lb]{\smash{{\SetFigFontNFSS{12}{14.4}{\rmdefault}{\mddefault}{\updefault}{\color[rgb]{0,0,0}$D$}%
}}}}
\put(8776,-3211){\makebox(0,0)[lb]{\smash{{\SetFigFontNFSS{12}{14.4}{\rmdefault}{\mddefault}{\updefault}{\color[rgb]{0,0,0}$A$}%
}}}}
\put(8686,-2536){\makebox(0,0)[lb]{\smash{{\SetFigFontNFSS{12}{14.4}{\rmdefault}{\mddefault}{\updefault}{\color[rgb]{0,0,0}$G$}%
}}}}
\put(7021,-3166){\makebox(0,0)[lb]{\smash{{\SetFigFontNFSS{12}{14.4}{\rmdefault}{\mddefault}{\updefault}{\color[rgb]{0,0,0}$H$}%
}}}}
\put(9631,-3121){\makebox(0,0)[lb]{\smash{{\SetFigFontNFSS{12}{14.4}{\rmdefault}{\mddefault}{\updefault}{\color[rgb]{0,0,0}$I$}%
}}}}
\put(10666,-2446){\makebox(0,0)[lb]{\smash{{\SetFigFontNFSS{12}{14.4}{\rmdefault}{\mddefault}{\updefault}{\color[rgb]{0,0,0}$J$}%
}}}}
\put(7741,-1771){\makebox(0,0)[lb]{\smash{{\SetFigFontNFSS{12}{14.4}{\rmdefault}{\mddefault}{\updefault}{\color[rgb]{0,0,0}$E$}%
}}}}
\put(10351,-1816){\makebox(0,0)[lb]{\smash{{\SetFigFontNFSS{12}{14.4}{\rmdefault}{\mddefault}{\updefault}{\color[rgb]{0,0,0}$F$}%
}}}}
\put(8776,-5191){\makebox(0,0)[lb]{\smash{{\SetFigFontNFSS{12}{14.4}{\rmdefault}{\mddefault}{\updefault}{\color[rgb]{0,0,0}$P_{1}$}%
}}}}
\put(10081,-4246){\makebox(0,0)[lb]{\smash{{\SetFigFontNFSS{12}{14.4}{\rmdefault}{\mddefault}{\updefault}{\color[rgb]{0,0,0}$\eta_3$}%
}}}}
\put(10576,-3571){\makebox(0,0)[lb]{\smash{{\SetFigFontNFSS{12}{14.4}{\rmdefault}{\mddefault}{\updefault}{\color[rgb]{0,0,0}$\eta_4$}%
}}}}
\put(8056,-4606){\makebox(0,0)[lb]{\smash{{\SetFigFontNFSS{12}{14.4}{\rmdefault}{\mddefault}{\updefault}{\color[rgb]{0,0,0}$\eta_2$}%
}}}}
\put(7381,-4111){\makebox(0,0)[lb]{\smash{{\SetFigFontNFSS{12}{14.4}{\rmdefault}{\mddefault}{\updefault}{\color[rgb]{0,0,0}$\eta_1$}%
}}}}
\put(8281,-1501){\makebox(0,0)[lb]{\smash{{\SetFigFontNFSS{12}{14.4}{\rmdefault}{\mddefault}{\updefault}{\color[rgb]{0,0,0}$\eta_4$}%
}}}}
\put(9811,-1456){\makebox(0,0)[lb]{\smash{{\SetFigFontNFSS{12}{14.4}{\rmdefault}{\mddefault}{\updefault}{\color[rgb]{0,0,0}$\eta_2$}%
}}}}
\end{picture}%

%% file: eg3.pstex_t
\begin{picture}(0,0)%
\includegraphics{eg3.pstex}%
\end{picture}%
\setlength{\unitlength}{4144sp}%
\begingroup\makeatletter\ifx\SetFigFontNFSS\undefined%
\gdef\SetFigFontNFSS#1#2#3#4#5{%
  \reset@font\fontsize{#1}{#2pt}%
  \fontfamily{#3}\fontseries{#4}\fontshape{#5}%
  \selectfont}%
\fi\endgroup%
\begin{picture}(7140,3063)(3541,-4675)
\put(4996,-4156){\makebox(0,0)[lb]{\smash{{\SetFigFontNFSS{12}{14.4}{\rmdefault}{\mddefault}{\updefault}{\color[rgb]{0,0,0}$O$}%
}}}}
\put(4546,-3166){\makebox(0,0)[lb]{\smash{{\SetFigFontNFSS{12}{14.4}{\rmdefault}{\mddefault}{\updefault}{\color[rgb]{0,0,0}$A$}%
}}}}
\put(4321,-3571){\makebox(0,0)[lb]{\smash{{\SetFigFontNFSS{12}{14.4}{\rmdefault}{\mddefault}{\updefault}{\color[rgb]{0,0,0}$B$}%
}}}}
\put(5761,-3571){\makebox(0,0)[lb]{\smash{{\SetFigFontNFSS{12}{14.4}{\rmdefault}{\mddefault}{\updefault}{\color[rgb]{0,0,0}$C$}%
}}}}
\put(5536,-3121){\makebox(0,0)[lb]{\smash{{\SetFigFontNFSS{12}{14.4}{\rmdefault}{\mddefault}{\updefault}{\color[rgb]{0,0,0}$D$}%
}}}}
\put(4096,-1861){\makebox(0,0)[lb]{\smash{{\SetFigFontNFSS{12}{14.4}{\rmdefault}{\mddefault}{\updefault}{\color[rgb]{0,0,0}$G$}%
}}}}
\put(3556,-2896){\makebox(0,0)[lb]{\smash{{\SetFigFontNFSS{12}{14.4}{\rmdefault}{\mddefault}{\updefault}{\color[rgb]{0,0,0}$H$}%
}}}}
\put(6346,-2896){\makebox(0,0)[lb]{\smash{{\SetFigFontNFSS{12}{14.4}{\rmdefault}{\mddefault}{\updefault}{\color[rgb]{0,0,0}$I$}%
}}}}
\put(5896,-1861){\makebox(0,0)[lb]{\smash{{\SetFigFontNFSS{12}{14.4}{\rmdefault}{\mddefault}{\updefault}{\color[rgb]{0,0,0}$J$}%
}}}}
\put(9361,-3931){\makebox(0,0)[lb]{\smash{{\SetFigFontNFSS{12}{14.4}{\rmdefault}{\mddefault}{\updefault}{\color[rgb]{0,0,0}$C_{1}$}%
}}}}
\put(7921,-3886){\makebox(0,0)[lb]{\smash{{\SetFigFontNFSS{12}{14.4}{\rmdefault}{\mddefault}{\updefault}{\color[rgb]{0,0,0}$B_{1}$}%
}}}}
\put(9631,-3121){\makebox(0,0)[lb]{\smash{{\SetFigFontNFSS{12}{14.4}{\rmdefault}{\mddefault}{\updefault}{\color[rgb]{0,0,0}$I_{1}$}%
}}}}
\put(10081,-3391){\makebox(0,0)[lb]{\smash{{\SetFigFontNFSS{12}{14.4}{\rmdefault}{\mddefault}{\updefault}{\color[rgb]{0,0,0}$D_{1}$}%
}}}}
\put(10666,-2446){\makebox(0,0)[lb]{\smash{{\SetFigFontNFSS{12}{14.4}{\rmdefault}{\mddefault}{\updefault}{\color[rgb]{0,0,0}$J_{1}$}%
}}}}
\put(10351,-1816){\makebox(0,0)[lb]{\smash{{\SetFigFontNFSS{12}{14.4}{\rmdefault}{\mddefault}{\updefault}{\color[rgb]{0,0,0}$F_{1}$}%
}}}}
\put(8776,-3211){\makebox(0,0)[lb]{\smash{{\SetFigFontNFSS{12}{14.4}{\rmdefault}{\mddefault}{\updefault}{\color[rgb]{0,0,0}$A_{1}$}%
}}}}
\put(8686,-2536){\makebox(0,0)[lb]{\smash{{\SetFigFontNFSS{12}{14.4}{\rmdefault}{\mddefault}{\updefault}{\color[rgb]{0,0,0}$G_{1}$}%
}}}}
\put(7021,-3166){\makebox(0,0)[lb]{\smash{{\SetFigFontNFSS{12}{14.4}{\rmdefault}{\mddefault}{\updefault}{\color[rgb]{0,0,0}$H_{1}$}%
}}}}
\put(7741,-1771){\makebox(0,0)[lb]{\smash{{\SetFigFontNFSS{12}{14.4}{\rmdefault}{\mddefault}{\updefault}{\color[rgb]{0,0,0}$E_{1}$}%
}}}}
\put(8551,-4516){\makebox(0,0)[lb]{\smash{{\SetFigFontNFSS{12}{14.4}{\rmdefault}{\mddefault}{\updefault}{\color[rgb]{0,0,0}$P_{1}^{\prime}$}%
}}}}
\put(5041,-4606){\makebox(0,0)[lb]{\smash{{\SetFigFontNFSS{12}{14.4}{\rmdefault}{\mddefault}{\updefault}{\color[rgb]{0,0,0}$\widetilde{P}_{1}$}%
}}}}
\end{picture}%

%% file: eg4.pstex_t
\begin{picture}(0,0)%
\includegraphics{eg4.pstex}%
\end{picture}%
\setlength{\unitlength}{4144sp}%
\begingroup\makeatletter\ifx\SetFigFontNFSS\undefined%
\gdef\SetFigFontNFSS#1#2#3#4#5{%
  \reset@font\fontsize{#1}{#2pt}%
  \fontfamily{#3}\fontseries{#4}\fontshape{#5}%
  \selectfont}%
\fi\endgroup%
\begin{picture}(9615,5358)(1336,-9490)
\put(1351,-7396){\makebox(0,0)[lb]{\smash{{\SetFigFontNFSS{12}{14.4}{\rmdefault}{\mddefault}{\updefault}{\color[rgb]{0,0,0}$O$}%
}}}}
\put(3196,-5911){\makebox(0,0)[lb]{\smash{{\SetFigFontNFSS{12}{14.4}{\rmdefault}{\mddefault}{\updefault}{\color[rgb]{0,0,0}$A$}%
}}}}
\put(1981,-6406){\makebox(0,0)[lb]{\smash{{\SetFigFontNFSS{12}{14.4}{\rmdefault}{\mddefault}{\updefault}{\color[rgb]{0,0,0}$B$}%
}}}}
\put(2026,-7936){\makebox(0,0)[lb]{\smash{{\SetFigFontNFSS{12}{14.4}{\rmdefault}{\mddefault}{\updefault}{\color[rgb]{0,0,0}$C$}%
}}}}
\put(3331,-7486){\makebox(0,0)[lb]{\smash{{\SetFigFontNFSS{12}{14.4}{\rmdefault}{\mddefault}{\updefault}{\color[rgb]{0,0,0}$D$}%
}}}}
\put(3646,-4876){\makebox(0,0)[lb]{\smash{{\SetFigFontNFSS{12}{14.4}{\rmdefault}{\mddefault}{\updefault}{\color[rgb]{0,0,0}$E$}%
}}}}
\put(3691,-8341){\makebox(0,0)[lb]{\smash{{\SetFigFontNFSS{12}{14.4}{\rmdefault}{\mddefault}{\updefault}{\color[rgb]{0,0,0}$F$}%
}}}}
\put(4186,-5146){\makebox(0,0)[lb]{\smash{{\SetFigFontNFSS{12}{14.4}{\rmdefault}{\mddefault}{\updefault}{\color[rgb]{0,0,0}$G$}%
}}}}
\put(2701,-5416){\makebox(0,0)[lb]{\smash{{\SetFigFontNFSS{12}{14.4}{\rmdefault}{\mddefault}{\updefault}{\color[rgb]{0,0,0}$H$}%
}}}}
\put(2791,-8431){\makebox(0,0)[lb]{\smash{{\SetFigFontNFSS{12}{14.4}{\rmdefault}{\mddefault}{\updefault}{\color[rgb]{0,0,0}$I$}%
}}}}
\put(4231,-7486){\makebox(0,0)[lb]{\smash{{\SetFigFontNFSS{12}{14.4}{\rmdefault}{\mddefault}{\updefault}{\color[rgb]{0,0,0}$J$}%
}}}}
\put(7336,-7396){\makebox(0,0)[lb]{\smash{{\SetFigFontNFSS{12}{14.4}{\rmdefault}{\mddefault}{\updefault}{\color[rgb]{0,0,0}$O$}%
}}}}
\put(8506,-5461){\makebox(0,0)[lb]{\smash{{\SetFigFontNFSS{12}{14.4}{\rmdefault}{\mddefault}{\updefault}{\color[rgb]{0,0,0}$H$}%
}}}}
\put(8686,-8656){\makebox(0,0)[lb]{\smash{{\SetFigFontNFSS{12}{14.4}{\rmdefault}{\mddefault}{\updefault}{\color[rgb]{0,0,0}$I$}%
}}}}
\put(4456,-8206){\makebox(0,0)[lb]{\smash{{\SetFigFontNFSS{12}{14.4}{\rmdefault}{\mddefault}{\updefault}{\color[rgb]{0,0,0}$I_{1}$}%
}}}}
\put(5086,-7756){\makebox(0,0)[lb]{\smash{{\SetFigFontNFSS{12}{14.4}{\rmdefault}{\mddefault}{\updefault}{\color[rgb]{0,0,0}$F_{1}$}%
}}}}
\put(5221,-7306){\makebox(0,0)[lb]{\smash{{\SetFigFontNFSS{12}{14.4}{\rmdefault}{\mddefault}{\updefault}{\color[rgb]{0,0,0}$J_{1}$}%
}}}}
\put(5311,-4651){\makebox(0,0)[lb]{\smash{{\SetFigFontNFSS{12}{14.4}{\rmdefault}{\mddefault}{\updefault}{\color[rgb]{0,0,0}$G_{1}$}%
}}}}
\put(4906,-4291){\makebox(0,0)[lb]{\smash{{\SetFigFontNFSS{12}{14.4}{\rmdefault}{\mddefault}{\updefault}{\color[rgb]{0,0,0}$E_{1}$}%
}}}}
\put(4321,-4471){\makebox(0,0)[lb]{\smash{{\SetFigFontNFSS{12}{14.4}{\rmdefault}{\mddefault}{\updefault}{\color[rgb]{0,0,0}$H_{1}$}%
}}}}
\put(9856,-8611){\makebox(0,0)[lb]{\smash{{\SetFigFontNFSS{12}{14.4}{\rmdefault}{\mddefault}{\updefault}{\color[rgb]{0,0,0}$I_{1}$}%
}}}}
\put(10936,-8116){\makebox(0,0)[lb]{\smash{{\SetFigFontNFSS{12}{14.4}{\rmdefault}{\mddefault}{\updefault}{\color[rgb]{0,0,0}$F_{1}$}%
}}}}
\put(10351,-7486){\makebox(0,0)[lb]{\smash{{\SetFigFontNFSS{12}{14.4}{\rmdefault}{\mddefault}{\updefault}{\color[rgb]{0,0,0}$J_{1}$}%
}}}}
\put(10351,-6091){\makebox(0,0)[lb]{\smash{{\SetFigFontNFSS{12}{14.4}{\rmdefault}{\mddefault}{\updefault}{\color[rgb]{0,0,0}$G_{1}$}%
}}}}
\put(9676,-4921){\makebox(0,0)[lb]{\smash{{\SetFigFontNFSS{12}{14.4}{\rmdefault}{\mddefault}{\updefault}{\color[rgb]{0,0,0}$H_{1}$}%
}}}}
\put(10846,-4921){\makebox(0,0)[lb]{\smash{{\SetFigFontNFSS{12}{14.4}{\rmdefault}{\mddefault}{\updefault}{\color[rgb]{0,0,0}$E_{1}$}%
}}}}
\put(7021,-7981){\makebox(0,0)[lb]{\smash{{\SetFigFontNFSS{12}{14.4}{\rmdefault}{\mddefault}{\updefault}{\color[rgb]{0,0,0}$\eta_2$}%
}}}}
\put(7066,-6631){\makebox(0,0)[lb]{\smash{{\SetFigFontNFSS{12}{14.4}{\rmdefault}{\mddefault}{\updefault}{\color[rgb]{0,0,0}$\eta_4$}%
}}}}
\put(8506,-4831){\makebox(0,0)[lb]{\smash{{\SetFigFontNFSS{12}{14.4}{\rmdefault}{\mddefault}{\updefault}{\color[rgb]{0,0,0}$\eta_1$}%
}}}}
\put(7381,-8521){\makebox(0,0)[lb]{\smash{{\SetFigFontNFSS{12}{14.4}{\rmdefault}{\mddefault}{\updefault}{\color[rgb]{0,0,0}$\eta_3$}%
}}}}
\put(9226,-9016){\makebox(0,0)[lb]{\smash{{\SetFigFontNFSS{12}{14.4}{\rmdefault}{\mddefault}{\updefault}{\color[rgb]{0,0,0}$\eta_2 + \eta_1$}%
}}}}
\put(10801,-8701){\makebox(0,0)[lb]{\smash{{\SetFigFontNFSS{12}{14.4}{\rmdefault}{\mddefault}{\updefault}{\color[rgb]{0,0,0}$\eta_2 + 2\eta_1$}%
}}}}
\put(8911,-9421){\makebox(0,0)[lb]{\smash{{\SetFigFontNFSS{12}{14.4}{\rmdefault}{\mddefault}{\updefault}{\color[rgb]{0,0,0}$P_2$}%
}}}}
\put(3151,-8836){\makebox(0,0)[lb]{\smash{{\SetFigFontNFSS{12}{14.4}{\rmdefault}{\mddefault}{\updefault}{\color[rgb]{0,0,0}$P_2$}%
}}}}
\end{picture}%

%% file: eg05.pstex_t
\begin{picture}(0,0)%
\includegraphics{eg05.pstex}%
\end{picture}%
\setlength{\unitlength}{4144sp}%
\begingroup\makeatletter\ifx\SetFigFontNFSS\undefined%
\gdef\SetFigFontNFSS#1#2#3#4#5{%
  \reset@font\fontsize{#1}{#2pt}%
  \fontfamily{#3}\fontseries{#4}\fontshape{#5}%
  \selectfont}%
\fi\endgroup%
\begin{picture}(3893,3877)(6826,-5264)
\put(9136,-4786){\makebox(0,0)[lb]{\smash{{\SetFigFontNFSS{12}{14.4}{\rmdefault}{\mddefault}{\updefault}{\color[rgb]{0,0,0}$O$}%
}}}}
\put(8686,-2536){\makebox(0,0)[lb]{\smash{{\SetFigFontNFSS{12}{14.4}{\rmdefault}{\mddefault}{\updefault}{\color[rgb]{0,0,0}$G$}%
}}}}
\put(7021,-3166){\makebox(0,0)[lb]{\smash{{\SetFigFontNFSS{12}{14.4}{\rmdefault}{\mddefault}{\updefault}{\color[rgb]{0,0,0}$H$}%
}}}}
\put(9631,-3121){\makebox(0,0)[lb]{\smash{{\SetFigFontNFSS{12}{14.4}{\rmdefault}{\mddefault}{\updefault}{\color[rgb]{0,0,0}$I$}%
}}}}
\put(10666,-2446){\makebox(0,0)[lb]{\smash{{\SetFigFontNFSS{12}{14.4}{\rmdefault}{\mddefault}{\updefault}{\color[rgb]{0,0,0}$J$}%
}}}}
\put(7741,-1771){\makebox(0,0)[lb]{\smash{{\SetFigFontNFSS{12}{14.4}{\rmdefault}{\mddefault}{\updefault}{\color[rgb]{0,0,0}$E$}%
}}}}
\put(10351,-1816){\makebox(0,0)[lb]{\smash{{\SetFigFontNFSS{12}{14.4}{\rmdefault}{\mddefault}{\updefault}{\color[rgb]{0,0,0}$F$}%
}}}}
\put(8776,-5191){\makebox(0,0)[lb]{\smash{{\SetFigFontNFSS{12}{14.4}{\rmdefault}{\mddefault}{\updefault}{\color[rgb]{0,0,0}$P^{\prime \prime}$}%
}}}}
\put(7921,-3886){\makebox(0,0)[lb]{\smash{{\SetFigFontNFSS{12}{14.4}{\rmdefault}{\mddefault}{\updefault}{\color[rgb]{0,0,0}$A$}%
}}}}
\put(9451,-3841){\makebox(0,0)[lb]{\smash{{\SetFigFontNFSS{12}{14.4}{\rmdefault}{\mddefault}{\updefault}{\color[rgb]{0,0,0}$B$}%
}}}}
\put(10081,-3391){\makebox(0,0)[lb]{\smash{{\SetFigFontNFSS{12}{14.4}{\rmdefault}{\mddefault}{\updefault}{\color[rgb]{0,0,0}$C$}%
}}}}
\put(8776,-3211){\makebox(0,0)[lb]{\smash{{\SetFigFontNFSS{12}{14.4}{\rmdefault}{\mddefault}{\updefault}{\color[rgb]{0,0,0}$D$}%
}}}}
\put(8956,-1546){\makebox(0,0)[lb]{\smash{{\SetFigFontNFSS{12}{14.4}{\rmdefault}{\mddefault}{\updefault}{\color[rgb]{0,0,0}$\eta_1$}%
}}}}
\put(6841,-2716){\makebox(0,0)[lb]{\smash{{\SetFigFontNFSS{12}{14.4}{\rmdefault}{\mddefault}{\updefault}{\color[rgb]{0,0,0}$\eta_3$}%
}}}}
\put(7381,-4111){\makebox(0,0)[lb]{\smash{{\SetFigFontNFSS{12}{14.4}{\rmdefault}{\mddefault}{\updefault}{\color[rgb]{0,0,0}$\eta_4$}%
}}}}
\put(8056,-4606){\makebox(0,0)[lb]{\smash{{\SetFigFontNFSS{12}{14.4}{\rmdefault}{\mddefault}{\updefault}{\color[rgb]{0,0,0}$\eta_1$}%
}}}}
\put(10576,-3571){\makebox(0,0)[lb]{\smash{{\SetFigFontNFSS{12}{14.4}{\rmdefault}{\mddefault}{\updefault}{\color[rgb]{0,0,0}$\eta_3$}%
}}}}
\put(10081,-4246){\makebox(0,0)[lb]{\smash{{\SetFigFontNFSS{12}{14.4}{\rmdefault}{\mddefault}{\updefault}{\color[rgb]{0,0,0}$\eta_2$}%
}}}}
\end{picture}%